\documentclass[12pt]{amsart}
\usepackage[utf8]{inputenc}
\usepackage[margin=2.5cm]{geometry}
\usepackage{graphicx}
\usepackage{amsmath}
\usepackage{amsfonts}
\usepackage{amssymb}
\usepackage{amsthm}
\usepackage{tikz-cd}
\usepackage{mathtools}
\usepackage{hyperref}

\theoremstyle{plain}
\newtheorem{theorem}{Theorem}[section]
\newtheorem{proposition}[theorem]{Proposition}
\newtheorem{lemma}[theorem]{Lemma}
\newtheorem{corollary}[theorem]{Corollary}

\theoremstyle{definition}
\newtheorem{definition}[theorem]{Definition}
\newtheorem{notation}[theorem]{Notation}
\newtheorem{question}[theorem]{Question}
\newtheorem{example}[theorem]{Example}
\newtheorem{remark}[theorem]{Remark}
\newtheorem{claim}{Claim}
\newtheorem*{IH}{Inductive hypothesis}

\makeatletter
    \@addtoreset{claim}{theorem}
    \@addtoreset{claim}{proposition}
    \@addtoreset{claim}{lemma}
    \@addtoreset{claim}{corollary}
\makeatother

\numberwithin{equation}{section}

\def\g{\mathfrak{g}}
\def\h{\mathfrak{h}}

\def\L{\mathcal{L}}
\def\T{\mathcal{T}}
\def\O{\mathcal{O}}
\def\Q{\mathcal{Q}}

\def\NN{\mathbb{N}}
\def\ZZ{\mathbb{Z}}

\def\Aa{\mathbb{A}}
\def\PP{\mathbb{P}}
\def\WW{\mathbb{W}}
\def\kk{\Bbbk}

\DeclareMathOperator{\Der}{Der}

\DeclareMathOperator{\spn}{span}
\DeclareMathOperator{\ord}{ord}
\DeclareMathOperator{\gr}{gr}
\DeclareMathOperator{\End}{End}
\DeclareMathOperator{\Aut}{Aut}
\DeclareMathOperator{\im}{Im}
\DeclareMathOperator{\ad}{ad}
\DeclareMathOperator{\Ver}{Ver}

\def\del{\partial}
\def\nonzero{\setminus \{0\}}
\def\ideal{\unlhd}
\newcommand{\diff}[1]{\frac{d}{d#1}}

\newcommand\restr[2]{{\left.\kern-\nulldelimiterspace #1 \vphantom{\big|} \right|_{#2}}}

\newcommand{\etalchar}[1]{$^{#1}$}

\title{Lie subalgebras of vector fields on curves}

\author{Lucas Buzaglo}
\author{Colin Ingalls}

\address[Buzaglo]{Department of Mathematics, UC San Diego, La Jolla, CA 92093-0112, USA}
\email{\href{mailto:lbuzaglo@ucsd.edu}{lbuzaglo@ucsd.edu}}

\address[Ingalls]{School of Mathematics and Statistics, Carleton University, Ottawa, ON K1S 5B6, Canada}
\email{\href{mailto:ColinIngalls@cunet.carleton.ca}{ColinIngalls@cunet.carleton.ca}}


\begin{document}

\begin{abstract}
    We study the subalgebra structure of Krichever--Novikov algebras, which are Lie algebras of vector fields on smooth affine curves. Our main result is that every infinite-dimensional subalgebra of a Krichever--Novikov algebra is isomorphic to a finite-codimensional subalgebra of another Krichever--Novikov algebra.
    
    We then present some applications of our main result. First, we show that the universal enveloping algebra of any such infinite-dimensional subalgebra is not noetherian. We then prove that all Krichever--Novikov algebras satisfy the Dixmier property that all their nonzero endomorphisms are automorphisms, except for the Witt algebra of vector fields on the once-punctured affine line. Finally, we provide an explicit classification of the infinite-dimensional subalgebras of the Witt algebra.
\end{abstract}

\maketitle

\section{Introduction}

Unlike their finite-dimensional counterparts, the general structure of infinite-dimensional Lie algebras remains mysterious. Indeed, there have been many attempts to classify certain classes of infinite-dimensional Lie algebras, usually requiring rather strong assumptions. For example, the $\ZZ$-graded simple Lie algebras of polynomial growth were classified in Mathieu's seminal works \cite{Mathieu1, Mathieu2, Mathieu3}. Another example of a classification result is \cite{Fialowski}, again with strong assumptions on the Lie algebras being classified.

One of the most important examples of an infinite-dimensional Lie algebra is the \emph{one-sided Witt algebra} $\WW_1$, which is the Lie algebra of vector fields on $\Aa^1$. Not only is this Lie algebra important mathematically, appearing in both Mathieu's and Fialowski's classification results, but it is also highly relevant in physics, with applications in conformal field theory and string theory. More generally, one can consider Lie algebras of vector fields on other (smooth) affine curves; these are known as \emph{Krichever--Novikov algebras} \cite{KricheverNovikov, Schlichenmaier}. While the local structure of Lie algebras of vector fields can be studied using formal algebraic models (see, for example, \cite{GuilleminSternberg}), their global structure is dependent on the underlying geometry. As noted in \cite{GuilleminSternberg}, ``the problem of describing the global possibilities for a fixed geometrical structure is a problem of an entirely different order and is usually extremely interesting and difficult''. Since Krichever--Novikov algebras are defined globally, we must use algebro-geometric techniques to bridge this gap.

The paper \cite{BellBuzaglo} gives a surprising and elegant classification of subalgebras of $\WW_1$: every infinite-dimensional subalgebra of $\WW_1$ is isomorphic to a finite-codimensional subalgebra of $\WW_1$ itself. Inspired by this classification result, we study subalgebras of Krichever--Novikov algebras over an algebraically closed field $\kk$ of characteristic zero, which we denote by $\L_C$, where $C$ is a smooth affine curve. Our main result completely generalizes \cite[Corollary 4.1]{BellBuzaglo}: we prove that every infinite-dimensional subalgebra of a Krichever--Novikov algebra is, up to finite codimension, another Krichever--Novikov algebra.

\begin{theorem}[Theorem~\ref{thm:main}]\label{thm:intro main}
    Let $C$ be a smooth affine curve and let $\g$ be an infinite-dimensional Lie subalgebra of $\L_C$. Then there exists another smooth affine curve $D$, defined canonically from the subalgebra $\g$, such that $\g$ embeds in $\L_D$ with finite codimension.
\end{theorem}

The new curve $D$ is defined from the subalgebra $\g$ by way of its \emph{field of ratios} $F(\g)$, which is a subfield of $\kk(C)$ (see Definition~\ref{def:ratios}). We define $\widetilde{D}$ to be a smooth projective curve whose function field is $F(\g)$. The affine curve $D$ is then defined to be the set of points of $\widetilde{D}$ where all vector fields in $\g$ are regular (meaning they have no poles).

The rest of the paper is devoted to presenting some applications of Theorem~\ref{thm:intro main}. One of the most noteworthy consequences is that all Krichever--Novikov algebras satisfy the \emph{Dixmier property} (see Definition~\ref{def:Dixmier}), except for the Witt algebra $W \coloneqq \L_{\Aa^1 \nonzero}$ of vector fields on the punctured affine line.

\begin{theorem}[Theorem~\ref{thm:Dixmier}]\label{thm:intro Dixmier}
    Let $C$ be a smooth affine curve. Then $\End(\L_C) \nonzero = \Aut(\L_C)$ if and only if $C \not\cong \Aa^1 \nonzero$.
\end{theorem}

The proof of Theorem~\ref{thm:intro Dixmier} is essentially achieved by showing that a nonzero endomorphism of $\L_C$ is induced by an \'etale self-map of $C$. It is not difficult to show that the only smooth affine curve which has nontrivial \'etale self-maps is $C = \Aa^1 \nonzero$: these are the maps $x \mapsto x^n$ for $n \geq 2$. However, proving that nonzero endomorphisms of $\L_C$ are induced by \'etale self-maps of $C$ is more challenging, requiring applications of Theorem~\ref{thm:intro main}, as well as deep results from \cite{Grabowski} and \cite{Siebert}.

We then explicitly classify subalgebras of the \emph{Witt algebra} $W \coloneqq \L_{\Aa^1 \nonzero}$, analogously to \cite[Theorem 2.8]{BellBuzaglo}. More precisely, given $s \in \kk[t,t^{-1}]$, we define the subalgebra
$$L(s) \coloneqq W \cap \Der(\kk[s])$$
of $W$, where the intersection is taken in $\Der(\kk(t)) = \kk(t)\del$ upon identifying $\Der(\kk[s]) = \frac{1}{s'}\kk[s]\del \subseteq \kk(t)\del$. We then show that all subalgebras of $W$ are either essentially of the form $L(s)$ for some $s \in \kk[t,t^{-1}]$, or they are contained in a Veronese subalgebra (see Definition~\ref{def:L(s)}).

\begin{theorem}[Theorem~\ref{thm:classification up to finite codimension}, Corollary~\ref{cor:sandwich}]\label{thm:intro classification}
    Let $\g$ be an infinite-dimensional subalgebra of $W$ which does not have finite codimension in any Veronese subalgebra of $W$. Then there exists $s \in \kk[t,t^{-1}]$ and a subalgebra $\h \subseteq \g$ which is also a $\kk[s]$-module such that
    $$\h \subseteq \g \subseteq L(s).$$
    In particular, $\g$ has finite codimension in $L(s)$.
\end{theorem}

Another application of Theorem~\ref{thm:intro main} is to the Sierra--Walton conjecture, which states that the universal enveloping algebra of a Lie algebra $\g$ is noetherian if and only if $\g$ is finite-dimensional \cite[Conjecture 0.1]{SierraWalton}, a long-standing question going back more than fifty years \cite{AmayoStewart}. It was shown in \cite{SierraWalton} that $U(\WW_1)$ is not noetherian, and \cite{Buzaglo} generalizes this to any Krichever--Novikov algebra, while \cite{BellBuzaglo2} further extends this to all Lie algebras of derivations of finitely generated associative algebras. Furthermore, it follows easily from the classification of subalgebras of $\WW_1$ that $U(\g)$ is not noetherian if $\g$ is an infinite-dimensional subalgebra of $\WW_1$ \cite[Corollary 4.2]{BellBuzaglo}. It is therefore natural to ask if the same can be done for subalgebras of arbitrary Krichever--Novikov algebras. In fact, this is an easy consequence of Theorem~\ref{thm:intro main} and results of \cite{Buzaglo}.

\begin{corollary}[Corollary~\ref{cor:noetherian}]
    Let $C$ be a smooth affine curve and let $\g$ be an infinite-dimensional Lie subalgebra of $\L_C$. Then $U(\g)$ is not noetherian.
\end{corollary}

Furthermore, the recent paper \cite{Mathieu} suggests that understanding the Krichever--Novikov algebras, particularly those of higher genus, might be important in the study of the Sierra--Walton conjecture. Of particular note are Theorem A and Section 9.2 of \cite{Mathieu}. Theorem~\ref{thm:intro main} establishes some strong structural properties satisfied by these Lie algebras.

Most of the paper is devoted to proving Theorem~\ref{thm:intro main}: in Section~\ref{sec:prelim}, we introduce some notation and basic results about Krichever--Novikov algebras. In Section~\ref{sec:new curve D}, we associate a new affine curve $D$ to an infinite-dimensional subalgebra $\g$ of $\L_C$ and prove that $\g$ canonically embeds in $\L_D$. In Section~\ref{sec:filtration}, we introduce a filtration on $\L_D$ which makes it easier to prove that $\g$ has finite codimension in $D$. Finally, Section~\ref{sec:proof of main theorem} completes the proof of Theorem~\ref{thm:intro main} by induction on the number of points at infinity of $D$. Section~\ref{sec:examples} presents examples of our construction of the curve $D$ for various subalgebras of Krichever--Novikov algebras. Finally, Theorems~\ref{thm:intro Dixmier} and \ref{thm:intro classification} are proved in Sections~\ref{sec:Dixmier} and \ref{sec:classification Witt}, respectively.

\subsection*{Acknowledgments}

This material is based upon work supported by the National Science Foundation under Grant No.~DMS-1928930 and by the Alfred P.~Sloan Foundation under grant G-2021-16778, while the authors were in residence at the Simons Laufer Mathematical Sciences Institute (formerly MSRI) in Berkeley, California, during the Spring 2024 semester. Much of this work was carried out at the SLMath workshop on infinite-dimensional division algebras, and at the Seattle Noncommutative Algebra Conference, both of which took place in December 2025. We thank the organizers of those conferences for the opportunities to discuss our work in person. The first-named author is supported by an AMS--Simons travel grant. The second-named author is supported by an NSERC discovery grant.

\section{Preliminaries}\label{sec:prelim}

Throughout this paper, we use the following conventions.
\begin{itemize}
    \item $\kk$ denotes an algebraically closed field of characteristic zero.
    \item $\kk^*$ denotes the set of nonzero elements of the field $\kk$.
    \item $\NN$ denotes the set of non-negative integers (including zero).
    \item $\kk(X)$ denotes the function field of a (projective or affine) variety $X$.
    \item $\O_X$ denotes the coordinate ring of an affine variety $X$.
    \item All Lie algebras and vector spaces are considered over $\kk$.
\end{itemize}
In this section, we give the main definitions and basic results that we will use throughout the rest of the paper. First, we recall the definition of Lie algebras of vector fields.

\begin{definition}
    Given a $\kk$-algebra $A$, a \emph{derivation} of $A$ is a $\kk$-linear map $d \colon A \to A$ satisfying the \emph{Leibniz rule}
    $$d(ab) = a d(b) + d(a)b$$
    for all $a,b \in A$. The set of all derivations of $A$ is denoted $\Der(A)$. It is a Lie algebra under the \emph{commutator bracket}
    $$[d_1,d_2] = d_1 \circ d_2 - d_2 \circ d_1,$$
    where $d_1,d_2 \in \Der(A)$.
    
    Let $X$ be an affine or projective variety. The \emph{Lie algebra of rational vector fields on $X$} is $\Q_X \coloneqq \Der(\kk(X))$. If $X$ is affine, then the \emph{Lie algebra of vector fields on $X$} is $\L_X \coloneqq \Der(\O_X)$.
\end{definition}

The Lie algebra of vector fields on an affine variety can equivalently be viewed as the space of sections of the tangent bundle of the variety. In other words, we can view $\L_X$ as $H^0(X,\T_X)$, where $X$ is an affine variety and $\T_X$ is its tangent bundle.

\begin{remark}
    If $C$ is a smooth affine curve, then the Lie algebra $\L_C$ is known as a \emph{Krichever--Novikov algebra}. See \cite{Schlichenmaier} and \cite{Buzaglo} for more details.
\end{remark}

Throughout, we fix $C$, an affine irreducible algebraic curve with smooth projective compactification $\widetilde{C}$. Of course, since $\kk(\widetilde{C}) = \kk(C)$, we have $\Q_{\widetilde{C}} = \Q_C$.

Let $p \in \widetilde{C}$. Locally at $p$, a vector field $v \in \Q_C$ has the form $v = f\del_t$, where $t$ is a uniformizer at $p$ and $\del_t$ denotes differentiation with respect to $t$. Here, $f$ is an element of $\kk((t))$, the field of fractions of the complete local ring $\widehat{\O}_{C,p} = \kk[[t]]$. From now on, we will say that ``$v = f\del_t$ locally at $p$'' to mean this without further comment. See the proof of \cite[Lemma 3.12]{Buzaglo} for details.

\begin{definition}
    Let $v \in \Q_C$ and let $t$ be a uniformizer at a point $p \in \widetilde{C}$. Let $v = f\del_t$ locally at $p$, and write
    $$f = \sum_{i = n}^\infty \alpha_i t^i,$$
    where $n \in \ZZ$, $\alpha_i \in \kk$ and $\alpha_n \neq 0$. The \emph{order of $v$ at $p$} is defined as
    $$\ord_p(v) = \ord_p(f) \coloneqq n,$$
    and the \emph{degree of $v$ at $p$} (or the \emph{$p$-degree of $v$}) is defined as
    $$\deg_p(v) \coloneqq 1 - \ord_p(v).$$
    We say that $v$ \emph{has a pole at $p$} if $\ord_p(v) < 0$ (equivalently, $\deg_p(v) \geq 2$), and that $v$ \emph{is regular at $p$} if $\ord_p(v) \geq 0$ (equivalently, $\deg_p(v) \leq 1$).
\end{definition}

We now demonstrate how the above definition works through an example.

\begin{example}
    Let $\WW_1 \coloneqq \L_{\Aa^1} = \Der(\kk[t]) = \kk[t]\del$ be the \emph{one-sided Witt algebra}, where $\del = \diff{t}$. In this case, we have $\widetilde{C} = \PP^1$. Letting $e_n \coloneqq t^{n + 1}\del$ be a standard basis element of $\WW_1$, we have $\deg_0(e_n) = -n$.
    
    To compute $\deg_\infty(e_n)$, we first need to write $e_n$ in terms of a uniformizer at $\infty$. To that end, let $s \coloneqq t^{-1}$. Then $\del_s = \frac{1}{s'}\del = -t^2\del$, and thus $\del = -t^{-2}\del_s = -s^2\del_s$. It follows that
    $$e_n = t^{n + 1}\del = s^{-n - 1}\del = -s^{-n + 1}\del_s.$$
    It follows that $\deg_\infty(e_n) = n$.
\end{example}

As the next result shows, the $p$-degree is additive on Lie brackets. In other words, the $p$-degree of the bracket of two vector fields is equal to the sum of their $p$-degrees, provided their degrees at $p$ are different.

\begin{lemma}\label{lem:degrees form a semigroup}
    Let $p \in \widetilde{C}$, and let $u,v \in \Q_C$. If $\deg_p(u) \neq \deg_p(v)$, then
    $$\deg_p([u,v]) = \deg_p(u) + \deg_p(v).$$
    If $\deg_p(u) = \deg_p(v)$, then $[u,v] = 0$ or $\deg_p([u,v]) < 2\deg_p(u)$.
\end{lemma}
\begin{proof}
    Let $t$ be a uniformizer at $p$ and write $u = f\del_t$ and $v = g\del_t$ locally at $p$.
    
    Let $n \coloneqq \deg_p(u)$ and $m \coloneqq \deg_p(v)$, and write $f = \sum_{k = 1 - n}^\infty \alpha_k t^k$ and $g = \sum_{k = 1 - m}^\infty \beta_k t^k$, where $\alpha_k, \beta_k \in \kk$ with $\alpha_{1 - n}, \beta_{1 - m} \neq 0$. We have
    $$[u,v] = [f\del_t, g\del_t] = (f\del_t(g) - \del_t(f)g)\del_t.$$
    It is easy to see that
    $$f\del_t(g) - \del_t(f)g = \Big((n - m)\alpha_{1 - n} \beta_{1 - m} t^{1 - (n + m)} + \mathrm{higher~degree~terms}\Big)\del_t.$$
    If $n = m$, then the above implies that $\deg_p([u,v]) < 2n$, as required. If $m \neq n$, then $t^{1 - (n + m)}$ has nonzero coefficient in $f\del_t(g) - \del_t(f)g$, which implies that
    $$\deg_p([u,v]) = n + m = \deg_p(u) + \deg_p(v).$$
    This concludes the proof.
\end{proof}

Taking repeated brackets between two vector fields with different degrees $p$, we are able to produce elements of large degrees at $p$, using Lemma~\ref{lem:degrees form a semigroup}. As we show next, starting with two vector fields of different degrees, this method can produce vector fields with any large enough degree at $p$, that is a multiple of the greatest common divisor of the two given vector fields.

\begin{corollary}\label{cor:gcd of degrees}
    Let $p \in \widetilde{C}$, let $u,v \in \Q_C$ such that $\deg_p(u) \neq \deg_p(v)$, and let $\g$ be the subalgebra of $\Q_C$ generated by $u$ and $v$. Letting $d \coloneqq \gcd(\deg_p(u),\deg_p(v))$, there exists $k \in \NN$ such that for all $\ell \geq k$, there in an element $w_\ell \in \g$ with $\deg_p(w_\ell) = \ell d$.
\end{corollary}
\begin{proof}
    Follows by the argument given in \cite[Lemma 4.7]{Buzaglo} using Lemma~\ref{lem:degrees form a semigroup}.
\end{proof}

As seen in Corollary~\ref{cor:gcd of degrees}, the greatest common divisor of the $p$-degrees in a subalgebra of $\L_C$ is a useful invariant. We now set some notation for this.

\begin{notation}
    Let $\g$ be an infinite-dimensional subalgebra of $\Q_C$. Letting $p \in \widetilde{C}$, we write
    $$d_p(\g) \coloneqq \gcd\{\deg_p(v) \mid v \in \g\}.$$
\end{notation}

We finish this section with a result about the $p$-degree of a linear combination of vector fields.

\begin{lemma}\label{lem:properties of deg(g)}
    Let $p \in \widetilde{C}$, and let $u,v \in \L_C$ be linearly independent elements. Then the following hold.
    \begin{enumerate}
        \item If $\deg_p(u) \neq \deg_p(v)$, then $\deg_p(\alpha u + \beta v) = \max(\deg_p(u),\deg_p(v))$ for all $\alpha, \beta \in \kk^*$.\label{item:maximum condition}
        \item If $\deg_p(u) = \deg_p(v)$, then there exist $\alpha,\beta \in \kk^*$ such that $\deg_p(\alpha u + \beta v) < \deg_p(u)$.\label{item:decrease equal degree}
    \end{enumerate}
\end{lemma}
\begin{proof}
    Let $t$ be a uniformizer at $p$ and view $u$ and $v$ as elements of $\kk((t))\del_t$. So, we can write
    $$u = (\lambda t^n + \text{higher degree terms})\del_t, \qquad v = (\mu t^m + \text{higher degree terms})\del_t,$$
    for some $\lambda,\mu \in \kk^*$ and $n,m \in \ZZ$. By definition, we have $\ord_p(u) = n$ and $\ord_p(v) = m$, so that $\deg_p(u) = 1 - n$ and $\deg_p(v) = 1 - m$.

    \eqref{item:maximum condition} Suppose $\deg_p(u) \neq \deg_p(v)$ (in other words, $n \neq m$). Without loss of generality, assume $n < m$. Letting $\alpha, \beta \in \kk^*$, we see that
    $$\alpha u + \beta v = (\alpha \lambda t^n + \text{higher degree terms})\del_t,$$
    and thus $\ord_p(\alpha u + \beta v) = \min(\ord_p(u),\ord_p(v))$. The result follows.

    \eqref{item:decrease equal degree} Suppose $\deg_p(u) = \deg_p(v)$ (in other words, $n = m$). Then we see that
    $$u - \frac{\lambda}{\mu} v = (\nu t^k + \text{higher degree terms})\del_t,$$
    for some $\nu \in \kk^*$ and $k \in \ZZ$ with $k > n$, since $u$ and $v$ are linearly independent. Therefore,
    $$\ord_p(\mu u - \lambda v) = k > n,$$
    which concludes the proof.
\end{proof}

\section{Associating curves to subalgebras}\label{sec:new curve D}

The main goal of this paper is to understand the subalgebra structure of $\L_C$, generalizing the classification of subalgebras of the \emph{one-sided Witt algebra} $\WW_1 \coloneqq \L_{\Aa^1} = \Der(\kk[t])$ of \cite{BellBuzaglo} to an arbitrary affine curve. The Lie algebra $\WW_1$ is a free $\kk[t]$-module of rank $1$: indeed, we have $\WW_1 = \kk[t]\del$ (where $\del = \diff{t}$). This is one of the key features of $\WW_1$ that the proof of \cite[Theorem 2.8]{BellBuzaglo} relies on, but is no longer true for an arbitrary affine curve $C$. In other words, $\L_C$ is not a free module over $\O_C$ in general. Thankfully, this is not difficult to fix: $\L_C$ is a locally free $\O_C$-module.

\begin{lemma}\label{lem:Q_C is one-dimensional}
    If $\del$ is any nonzero element of $\Der(\kk(C))$, then
    $$\Q_C = \Der(\kk(C)) = \kk(C)\del.$$
\end{lemma}
\begin{proof}
    It is well-known that the Lie algebra $\Q_C$ is a one-dimensional $\kk(C)$-vector space -- see, for example, \cite[Example 15.2.3]{McConnellRobson}.
\end{proof}

We now fix $\g$, an infinite-dimensional Lie subalgebra of $\L_C$, and a nonzero element $\del \in \L_C$. By Lemma~\ref{lem:Q_C is one-dimensional}, we have $\L_C \subseteq \Q_C = \kk(C)\del$ (where we view $\del$ as a derivation of $\kk(C)$). Therefore, any element of $\L_C$ can be written as $f\del$ for some $f \in \kk(C)$. The next definition is a generalization of \cite[Definition 5.2]{Buzaglo} to our more general situation.

\begin{definition}\label{def:ratios}
    We define the \emph{set of ratios} of $\g$ as
    $$R(\g) = \left\{\frac{u}{v} \in \kk(C) \mid u\del,v\del \in \g, v \neq 0\right\}.$$
    In other words, elements of the set of ratios of $\g$ are obtained by taking two elements $u\del, v\del$ of $\g$ with $v \neq 0$ (where $u,v \in \kk(C)$), and taking the ratio $\frac{u}{v}$.

    The \emph{field of ratios} of $\g$, denoted $F(\g)$, is defined to be the subfield of $\kk(C)$ generated by $R(\g)$.
\end{definition}

\begin{remark}
    We highlight that the definition of $R(\g)$ does not depend on the choice of $\del$. Indeed, if $\del' \in \L_C \nonzero$ is another choice, then $\del' \in \Q_C = \kk(C)\del$, and thus $\del' = f\del$ for some $f \in \kk(C)$. It is now clear that the set
    $$\left\{\frac{u}{v} \in \kk(C) \mid u\del', v\del' \in \g, v \neq 0\right\}$$
    is equal to $R(\g)$, since the extra factor of $f$ cancels in the ratio $\frac{u}{v}$.
\end{remark}

We now show that $\g$ can be embedded into the Lie algebra of derivations of its field of ratios $F(\g)$.

\begin{lemma}\label{lem:g acts on F(g)}
    We have $\g \subseteq \Der(F(\g))$.
\end{lemma}
\begin{proof}
    Let $u\del,v\del,w\del \in \g$, with $v \neq 0$, where $u,v,w \in \kk(C)$, so that $\frac{u}{v} \in R(\g)$. We claim that $w\del(\frac{u}{v}) \in F(\g)$, which means that $\g$ maps $R(\g)$ to $F(\g)$. Since $F(\g)$ is generated by $R(\g)$ as a field, it will then follow that $\g$ preserves $F(\g)$.

    We have
    $$w\del\left(\frac{u}{v}\right) = w \cdot \frac{v\del(u) - u\del(v)}{v^2} = \frac{w}{v} \cdot \frac{v\del(u) - u\del(v)}{v}.$$
    Certainly, $\frac{w}{v} \in R(\g)$ by definition. Note that $[v\del,u\del] = (v\del(u) - u\del(v))\del \in \g$, so we also have
    $$\frac{v\del(u) - u\del(v)}{v} \in R(\g).$$
    It follows that $w\del(\frac{u}{v}) \in F(\g)$, as claimed.
\end{proof}

Next, we use the field of ratios $F(\g)$ to define a projective curve associated to the subalgebra $\g$. We can then remove the points of this new curve where elements of $\g$ have poles to get an affine curve.

\begin{definition}\label{def:curves associated to g}
    The \emph{projective curve associated to $\g$} is the unique smooth projective curve $\widetilde{D}_\g$ such that $F(\g) \cong \kk(\widetilde{D}_\g)$. The \emph{affine curve associated to $\g$} is defined as
    $$D_\g \coloneqq \{p \in \widetilde{D} \mid \ord_p(w) \geq 0 \text{ for all } w \in \g\}.$$
    We will often drop the subscript and write $\widetilde{D}$ and $D$ instead of $\widetilde{D}_\g$ and $D_\g$ if there is no potential for confusion.
\end{definition}

Notice that $\g \subseteq \Q_{\widetilde{D}}$ by Lemma~\ref{lem:g acts on F(g)}, so $\g$ is a subalgebra of (rational) vector fields on $\widetilde{D}$. In fact, as we show next, $\g$ is a subalgebra of vector fields on $D$.

\begin{lemma}
    We have $\g \subseteq \L_{D}$.
\end{lemma}
\begin{proof}
    Let $p$ be a point in $D$. Note that $\g$ must preserve the Zariski local ring $\O_{D,p}$ since $\g$ does not have poles on points of $D$. Therefore, $\g$ preserves the intersection of all localizations at each point in $D$. In other words, $\g$ preserves $\bigcap_{p \in D} \O_{D,p} = \O_D$.
\end{proof}

We will now shift from viewing $\g$ as a subalgebra of vector fields on $C$ to viewing $\g$ as a subalgebra of vector fields on $D$. In some sense, $\L_{D}$ is the more natural space in which $\g$ lives. This is made precise in the next theorem, which is the main result of the paper. The proof will be completed in Section~\ref{sec:proof of main theorem}.

\begin{theorem}\label{thm:main}
    The subalgebra $\g$ has finite codimension in $\L_{D}$.
\end{theorem}

Theorem~\ref{thm:main} says that, even though $\g$ might be very far from $\L_C$ (in the sense that $\g$ might have infinite codimension in $\L_C$), it is very close to being the entire Lie algebra $\L_{D}$. In particular, Theorem~\ref{thm:main} implies that every infinite-dimensional subalgebra of a Krichever--Novikov algebra is isomorphic to a finite-codimensional subalgebra of another Krichever--Novikov algebra.

Using Theorem~\ref{thm:main}, it is easy to prove that the universal enveloping algebra of any infinite-dimensional subalgebra of $\L_C$ is non-noetherian, providing further evidence for the Sierra--Walton conjecture \cite[Conjecture 0.1]{SierraWalton}.

\begin{corollary}\label{cor:noetherian}
    Let $C$ be an affine curve, and let $\g$ be an infinite-dimensional subalgebra of $\L_C$. Then $U(\g)$ is not noetherian.
\end{corollary}
\begin{proof}
    Let $D \coloneqq D_\g$ be the affine curve associated to $\g$. By Theorem~\ref{thm:main}, $\g$ has finite codimension in $\L_D$. It is known that $U(\L_D)$ is not noetherian \cite[Theorem 3.3]{Buzaglo}. It now follows by \cite[Proposition 2.1]{Buzaglo} that $U(\g)$ is not noetherian.
\end{proof}

\begin{remark}
    The Informal Corollary of \cite{Mathieu} says that it essentially suffices to prove the Sierra--Walton conjecture \cite[Conjecture 0.1]{SierraWalton} for three types of Lie algebras: simple ones, residually nilpotent ones, and a third class of Lie algebras of which we do not currently have any examples. Krichever--Novikov algebras provide important examples of the first two types: they are simple \cite{Siebert}, while many of their subalgebras are residually nilpotent. Theorem~\ref{thm:intro main} shows that infinite-codimensional subalgebras of Krichever--Novikov algebras do not give any new examples of Lie algebras. In other words, it suffices to consider subalgebras of finite codimension.
\end{remark}

\section{Examples}\label{sec:examples}

In this section, we present examples of our construction applied to various subalgebras of Krichever--Novikov algebras.

\subsection{Some subalgebras of the Witt algebra}

First, we consider some subalgebras of the Witt algebra. The \emph{(full) Witt algebra} is $W \coloneqq \Der(\kk[t,t^{-1}]) = \kk[t,t^{-1}]\del$, where $\del \coloneqq \diff{t}$. Define $e_n \coloneqq t^{n + 1}\del$ for $n \in \ZZ$. The Lie bracket of this basis of $W$ is
$$[e_n,e_m] = (m - n)e_{n+m}.$$
For the Witt algebra, we have:
\begin{itemize}
    \item $C = \Aa^1 \setminus \{0\} = \PP^1 \setminus \{0, \infty\}$.
    \item $\widetilde{C} = \PP^1$.
\end{itemize}
The \emph{one-sided Witt algebra} is $\WW_1 \coloneqq \Der(\kk[t]) = \kk[t]\del$, which we view as the subalgebra of $W$ spanned by $e_n$ with $n \geq -1$. For the one-sided Witt algebra, we have:
\begin{itemize}
    \item $C = \Aa^1 = \PP^1 \setminus \{\infty\}$.
    \item $\widetilde{C} = \PP^1$.
\end{itemize}
For the examples in this section, we will need the notion of a \emph{submodule-subalgebra} of $W$ and $\WW_1$.

\begin{definition}
    A \emph{submodule-subalgebra} of $W$ or of $\WW_1$ is a Lie subalgebra which is also a $\kk[t,t^{-1}]$- or $\kk[t]$-submodule of $W$ or $\WW_1$, respectively. It is easy to see that they are of the form $fW = \kk[t,t^{-1}]f\del$ or $f\WW_1 = \kk[t]f\del$ for some $f \in \kk[t]$.
\end{definition}

We now proceed with our first example.

\begin{example}\label{ex:subalgebra of W1}
    Let $\g \coloneqq \kk[t^2]t\del = \spn\{e_0,e_2,e_4,\dots\}$, which is a subalgebra of $\WW_1$. It is easy to see that $\g$ has a pole at $\infty$ in $\widetilde{C} = \PP^1$. Write $s \coloneqq t^2$, so that $F(\g) = \kk(t^2) = \kk(s)$. For clarity, write $\PP^1_s$ to mean $\PP^1$, but with $\kk(\PP^1_s) = \kk(s)$, rather than $\kk(\PP^1) = \kk(t)$, so that $\widetilde{D} = \PP^1_s$.

    Note that $\Der(\kk[s]) = \kk[s]\del_s$, where $\del_s$ is the unique derivation of $\kk[s]$ such that $\del_s(s) = 1$. This derivation can be uniquely extended to a derivation of $\kk(t)$ by the identification
    $$\del_s = \frac{1}{s'}\del = \frac{1}{2t} \del.$$
    Under this identification, we have
    $$e_0 = t\del = 2t^2\del_s = 2s\del_s, \qquad e_2 = t^3\del = 2t^4\del_s = 2s^2\del_s, \qquad e_4 = t^5\del = 2t^6\del_s = 2s^3\del_s, \dots,$$
    so we see that $\g = \kk[s]s\del_s$. Hence, $\g$ only has a pole at $\infty$ in $\widetilde{D} = \PP^1_s$. Therefore, $D = \PP^1_s \setminus \{\infty\} = \Aa^1_s$, where, similarly to before, the notation $\Aa^1_s$ means $\Aa^1$ but with $\O_{\Aa^1_s} = \kk[s]$ instead of $\O_{\Aa^1} = \kk[t]$.
    
    It is now clear that $\g = \kk[t^2]t\del = \kk[s]s\del_s$ is a subalgebra of $\L_D = \Der(\kk[s]) = \kk[s]\del_s \cong \WW_1$ of codimension $1$, and that $\g \cong t\WW_1$.
\end{example}

The next example is similar to Example \ref{ex:subalgebra of W1}, except this time the subalgebra $\g$ has two poles in $\widetilde{C}$ instead of just one.

\begin{example}\label{ex:easy example}
    Consider the following subalgebra of $W$:
    $$\g \coloneqq \spn\{e_{-2},e_0,e_2,e_4,\ldots\} = \kk[t^2]t^{-1}\del.$$
    Note that the Lie algebra $\g$ has poles at $0$ and $\infty$ in $\widetilde{C} = \PP^1$: for example, the element $e_{-2} = t^{-1}\del \in \g$ has a pole at $0$, while $e_2 = t^3\del \in \g$ has a pole at $\infty$.

    As in Example \ref{ex:subalgebra of W1}, we let $s \coloneqq t^2$. Then $F(\g) = \kk(t^2)$, so $\widetilde{D} \cong \PP^1_s$. Proceeding as in Example \ref{ex:subalgebra of W1} to write $\g$ in terms of $s$, we can easily see that $\g = \kk[s]\del_s = \Der(\kk[s])$. Therefore, $\g$ only has a pole at $\infty$ in $\PP^1_s$. This means that the two poles $0$ and $\infty$ of $\g$ in $\widetilde{C}$ collapse to only one pole $\infty$ in $\widetilde{D}$. In other words, we have
    $$D = \PP^1_s \setminus \{\infty\} = \Aa^1_s.$$
    Note that in this case $\g$ is equal to $\L_D$.
\end{example}

We finish with a more complicated example.

\begin{example}\label{ex:harder example}
    Consider the following subalgebra of $W$:
    $$\g = \spn\{e_1 - e_{-1}, e_2 - e_{-2}, e_3 - e_{-3}, \ldots\}.$$
    Letting $\O_n \coloneqq e_n - e_{-n}$, one can easily check that
    $$[\O_n,\O_m] = (m - n)\O_{n + m} + (n + m)\O_{n - m}.$$
    This Lie algebra has appeared in \cite{BagchiChakrabortyChakraborttyFredenhagenGrumillerPandit, BCCA} in the context of Carrollian physics.
    
    As in Example \ref{ex:easy example}, $\g$ has poles at $0$ and $\infty$ in $\widetilde{C} = \PP^1$: for example, the element $e_2 - e_{-2} = (t^3 + t^{-1})\del \in \g$ has poles at both of these points. It is not too difficult to compute that $F(\g) = \kk(t + t^{-1})$ in this case. As we did in the previous example, let $s \coloneqq t + t^{-1}$, so that $\widetilde{D} = \PP^1_s$ (and thus $\kk(\widetilde{D}) = \kk(\PP^1_s) = \kk(s) = F(\g)$).
    
    We claim that $\g = \kk[s](s^2 - 4)\del_s$. Note that
    \begin{align*}
        (s^2 - 4)\del_s &= \frac{s^2 - 4}{s'}\del = \frac{(t + t^{-1})^2 - 4}{(t + t^{-1})'}\del = \frac{t^2 - 2 + t^{-2}}{1 - t^{-2}}\del = \frac{(1 - t^{-2})(t^2 - 1)}{1 - t^{-2}}\del \\
        &= (t^2 - 1)\del = e_1 - e_{-1} \in \g.
    \end{align*}
    The claim now follows easily by a similar computation. Now, $\g = \kk[s](s^2 - 4)\del_s$ only has a pole at $\infty \in \PP^1_s$. Once again, the two poles in $\widetilde{C}$ have collapsed to a single point in $\widetilde{D}$. Therefore,
    $$D = \PP^1_s \setminus \{\infty\} = \Aa^1_s,$$
    and thus
    $$\L_D = \L_{\Aa^1_s} = \Der(\kk[s]) = \kk[s]\del_s.$$
    We conclude that $\g = \kk[s](s^2 - 4)\del_s$ has finite codimension in $\L_D$ (it has codimension $2$), and that $\g \cong (t^2 - 4)\WW_1$.
\end{example}

\subsection{A subalgebra of vector fields on an elliptic curve}

We conclude this section with one more example in higher genus.

\begin{example}\label{ex:elliptic}
    Let $a, b, c \in \kk$ be pairwise distinct scalars with $a + b + c = 0$, and let $E$ be the (affine part of an) elliptic curve defined by $y^2 = 4(x - a)(x - b)(x - c)$ with the point $(a,0)$ removed. It was shown in \cite[Propositions 3 and 4]{SchlichenmaierDegenerations} that
    $$\L_E = \Der\left(\frac{\kk[x, y, (x - a)^{-1}]}{(y^2 - 4(x - a)(x - b)(x - c))}\right)$$
    is spanned by the vector fields
    $$v_{2n} \coloneqq 2(x - a)^{n - 1}(x - b)(x - c)\del_x, \qquad v_{2n + 1} \coloneqq (x - a)^n y \del_x,$$
    where $n \in \ZZ$. Letting $\alpha \coloneqq 3a$ and $\beta \coloneqq (a - b)(a - c) \in \kk$, the bracket of $\L_E$ is given by
    \begin{equation}\label{eq:elliptic Lie bracket}
        [v_n,v_m] = \begin{cases}
        (m - n)v_{n + m}, &n, m \text{ odd}, \\
        (m - n)(v_{n + m} + \alpha v_{n + m - 2} + \beta v_{n + m - 4}), &n, m \text{ even}, \\
        (m - n)v_{n + m} + (m - n + 1)\alpha v_{n + m - 2} + (m - n + 2)\beta v_{n + m - 4}, &n \text{ even}, m \text{ odd}.
    \end{cases}
    \end{equation}
    Let $\g \coloneqq \spn\{v_n \mid n \in 2\ZZ\}$, which is a subalgebra of $\L_E$ of infinite codimension. It is easy to see that $F(\g) = \kk(x)$, so $\widetilde{D} = \PP^1_x$. Note that $\g$ has poles at $x = a$ and at $x = \infty$ in $\PP^1_x$, so we see that $D = \PP^1_x \setminus \{a,\infty\} \cong \Aa^1 \setminus \{0\}$ and thus $\L_D \cong W$. Under this isomorphism, $\g$ corresponds to the subalgebra $(t + a - b)(t + a - c)W = (t^2 + \alpha t + \beta)W$ of $W$, which has codimension $2$.
\end{example}

\begin{remark}
    One could similarly study subalgebras of \emph{superelliptic algebras}, which are defined by an equation of the form
    $$y^m = \prod_{i = 1}^n (x - \alpha_i)$$
    where $n,m \geq 2$ and $\alpha_i \in \kk$ are pairwise distinct. Their Lie algebras of vector fields were studied in \cite{CoxGuoLuZhao}: for small values of $n$ and $m$, it should be possible to get an explicit Lie algebra structure like in \eqref{eq:elliptic Lie bracket}.
\end{remark}

\section{Poles in \texorpdfstring{$\widetilde{C}$}{C~} and in \texorpdfstring{$\widetilde{D}$}{D~}}\label{sec:relating poles in C and D}

We now start working toward a proof of Theorem~\ref{thm:main}. Since $\kk(\widetilde{D}) = F(\g) \subseteq \kk(\widetilde{C})$, we get a dominant (and therefore surjective) morphism $\pi \colon \widetilde{C} \to \widetilde{D}$. Also, note that any derivation of $\kk(\widetilde{D})$ extends uniquely to a derivation of $\kk(\widetilde{C})$, so we view $\Q_{\widetilde{D}}$ as a Lie subalgebra of $\Q_{\widetilde{C}}$. We now prove a result relating the degree of an element in $\Q_{\widetilde{D}}$ to its degree when viewed as an element of $\Q_{\widetilde{C}}$, via the ramification of the map $\pi$.

\begin{lemma}\label{lem:ramification}
    Let $q \in \widetilde{C}$, define $p \coloneqq \pi(q) \in \widetilde{D}$, and let $e$ be the ramification index of $\pi$ at $q$. Then for all $v \in \Q_{\widetilde{D}}$, we have
    $$\ord_q(v) = e(\ord_p(v) - 1) + 1, \text { or equivalently, } \deg_q(v) = e\deg_p(v).$$
\end{lemma}
\begin{proof}
    Let $v \in \Q_{\widetilde{D}}$. Choose uniformizers $t$ at $q \in \widetilde{C}$ and $s$ at $p \in \widetilde{D}$ such that $s = t^e$. Locally at $p$, the vector field $v$ has the form $f(s) \del_s$, where $f(s) \in \kk((s))$. Therefore, when viewed as an element of $\Q_{\widetilde{C}}$, the vector field $v$ has the following form (locally at the point $q$):
    $$v = f(t^e) \del_s = f(t^e) \frac{1}{\del_t(s)} \del_t = f(t^e) \frac{1}{e t^{e - 1}} \del_t = \frac{1}{e} f(t^e) t^{1 - e} \del_t,$$
    where we used that $\del_s = \frac{1}{\del_t(s)} \del_t$ by the chain rule. Therefore,
    $$\ord_q(v) = e \ord_p(v) + 1 - e = e(\ord_p(v) - 1) + 1.$$
    The result now follows upon recalling that $\deg_p(v) = 1 - \ord_p(v)$ by definition.
\end{proof}

Using Lemma~\ref{lem:ramification}, it is easy to prove that any preimage of a pole of $\g$ in $\widetilde{D}$ is a pole in $\widetilde{C}$.

\begin{lemma}\label{lem:poles in D come from poles in C}
    Let $p \in \widetilde{D} \setminus D$. Then $\pi^{-1}(p) \subseteq \widetilde{C} \setminus C$.
\end{lemma}
\begin{proof}
    Since $p \in \widetilde{D} \setminus D$, there exists $v \in \g$ such that $v$ has a pole at $p$. In other words, $\ord_p(v) < 0$. Let $q \in \pi^{-1}(p)$. Choose uniformizers $t$ at $q \in \widetilde{C}$ and $s$ at $p \in \widetilde{D}$ such that $s = t^e$, where $e$ is the ramification index of $\pi$ at $q$. By Lemma~\ref{lem:ramification}, we have
    $$\ord_q(v) = e(\ord_p(v) - 1) + 1 < 0,$$
    since $\ord_p(v) - 1 \leq -2$ by the assumption that $v$ has a pole at $p$. Therefore, $v$ has a pole at $q$. Since $v \in \L_C$, it must be the case that $q \notin C$, which concludes the proof.
\end{proof}

\begin{remark}
    It is not necessarily true that if $q \in \widetilde{C} \setminus C$ then $\pi(q) \in \widetilde{D} \setminus D$. For an explicit example, see Example \ref{ex:easy example}, where $\pi(0) = 0$ with $0 \in \widetilde{C} \setminus C$ but $0 \notin \widetilde{D} \setminus D$.
\end{remark}

Lemma~\ref{lem:poles in D come from poles in C} has the following consequence.

\begin{corollary}\label{cor:C maps to D}
    The morphism $\pi \colon \widetilde{C} \to \widetilde{D}$ restricts to a dominant morphism of affine curves $\pi' \colon C \to D$.
\end{corollary}
\begin{proof}
    Lemma~\ref{lem:poles in D come from poles in C} guarantees that $\pi(C) \subseteq D$.
\end{proof}

\begin{example}
    Let $E$ and $\g \subseteq \L_E$ be as in Example \ref{ex:elliptic}. In this case the map $\pi$ from Corollary~\ref{cor:C maps to D} is
    \begin{align*}
        \pi \colon E &\to \Aa^1 \setminus \{a\} \\
        (x,y) &\mapsto x.
    \end{align*}
\end{example}

The rest of this section is devoted to proving the following result.

\begin{proposition}\label{prop:g has almost all poles}
    We have
    $$d_p(\g) = \gcd\{\deg_p(v) \mid v \in \g\} = 1$$
    for all $p \in \widetilde{D} \setminus D$.
\end{proposition}

Proposition~\ref{prop:g has almost all poles} is the analog of \cite[Proposition 3.1]{BellBuzaglo} in our more general setup. Since the proof of \cite[Proposition 3.1]{BellBuzaglo} uses local arguments, we will prove Proposition~\ref{prop:g has almost all poles} in a similar way. For this reason, we omit many details and refer the reader to \cite{BellBuzaglo}. First, we require the following Hensel lemma argument analogous to \cite[Lemma 3.2]{BellBuzaglo}.

\begin{lemma}\label{lem:Hensel}
    Let $p \in \widetilde{D} \setminus D$, let $s$ be a uniformizer at $p$, and let $z = s^{-1}$. Given an element $v \in \g$ with a pole at $p$, there exist $r \in \kk((z^{-1}))$ with $r = z + \mathrm{lower~degree~terms}$ and $\alpha \in \kk^*$ such that
    $$v = \alpha r^d \del_r = \frac{\alpha r^d}{\del_z(r)}\del_z.$$
\end{lemma}
\begin{proof}
    Locally at $p$, we can write $v = v(z) \del_z$ as a derivation of $\kk((z^{-1}))$. Note that, since $v$ has a pole at $p$, we have $v(z) = \alpha z^d + \mathrm{lower~degree~terms} \in \kk((z^{-1}))$ for some $\alpha \in \kk^*$ and some positive integer $d$.

    The rest of the proof is identical to the proof of \cite[Lemma 3.2]{BellBuzaglo}: we can choose $r_0 = z$ and $r_{n + 1} = r_n - \frac{\beta}{\alpha (n + d)} z^{-n}$ for all $n \in \NN$.
\end{proof}

Next, we recall the definition of the complete Veronese subalgebras from \cite[Notation 3.3]{BellBuzaglo} and a result from \cite{BellBuzaglo}.

\begin{definition}
    Let $r \in \kk((z^{-1}))$ such that $r = z + \mathrm{lower~degree~terms}$, and let $d$ be a positive integer. The \emph{complete $d$-Veronese subalgebra with parameter $r$} is the subalgebra $V_d(r) \coloneqq \kk((r^{-d}))r\del_r$ of $\kk((z^{-1}))\del_z$.
\end{definition}

\begin{lemma}[{\cite[Corollary 3.5]{BellBuzaglo}}]\label{lem:contained in Veronese}
    Let $\h$ be an infinite-dimensional subalgebra of $\kk((z^{-1}))\del_z$, and let $r = z + \mathrm{lower~degree~terms} \in \kk((z^{-1}))$. Then $\h \cap V_d(r) \neq 0$ if and only if $\h \subseteq V_d(r)$.
\end{lemma}

The following result, which shows that $\g$ is contained in some complete Veronese subalgebra, is the last ingredient in the proof of Proposition~\ref{prop:g has almost all poles}.

\begin{corollary}\label{cor:contained in Veronese}
    Let $p \in \widetilde{D} \setminus D$ and write $d \coloneqq d_p(\g)$. Then $\g \subseteq V_d(r)$ for some $r = z + \mathrm{lower~degree~terms} \in \kk((z^{-1}))$.
\end{corollary}
\begin{proof}
    The proof is essentially the same as \cite[Corollary 3.6]{BellBuzaglo}: let $v$ be an element of $\g$ with a pole at $p$. Letting $s$ be a uniformizer at $p$ and $z \coloneqq s^{-1}$, Lemma~\ref{lem:Hensel} implies that, rescaling $v$ if necessary, there exists $r = z + \mathrm{lower~degree~terms} \in \kk((z^{-1}))$ such that $v = r^n \del_r$.
    
    We now compute $\ord_p(v)$. To do so, we must write $v$ in terms of $s$. Notice that
    $$r = z + \mathrm{lower~degree~terms} = s^{-1} + \mathrm{higher~degree~terms},$$
    where ``higher degree'' here refers to powers of $s$. Therefore,
    $$v = r^n \del_r = \frac{r^n}{\del_s(r)} \del_s = \frac{s^{-n} + \mathrm{higher}}{-s^{-2} + \mathrm{higher}} \del_s = -(s^{-n + 2} + \mathrm{higher}) \del_s.$$
    Hence, $\ord_p(v) = -n + 2$. By definition of $d_p(\g)$, it follows that $d$ divides
    $$\deg_p(v) = 1 - \ord_p(v) = n - 1.$$ 
    In other words, $n - 1 = kd$ for some $k \in \ZZ$. Therefore, $v = r^{kd + 1} \del_r \in \g$, so $\g \cap V_d(r) \neq 0$. We conclude that $\g \subseteq V_d(r)$, by Lemma~\ref{lem:contained in Veronese}.
\end{proof}

We are now ready to prove Proposition~\ref{prop:g has almost all poles}.

\begin{proof}[Proof of Proposition~\ref{prop:g has almost all poles}]
    Let $p \in \widetilde{D} \setminus D$, let $s$ be a uniformizer at $p$, and write $d \coloneqq d_p(\g)$. By Corollary~\ref{cor:contained in Veronese}, there exists $r = s^{-1} + \mathrm{higher~degree~terms} \in \kk((s))$ such that $\g \subseteq V_d(r) = \kk((r^{-d}))r\del_r$, so it follows that $\kk(D) = F(\g) \subseteq \kk((r^{-d}))$. In particular, since $s \in \kk(D)$, we see that $s \in \kk((r^{-d}))$. Therefore, $d = 1$.
\end{proof}

\section{Filtering \texorpdfstring{$\L_D$}{TD}}\label{sec:filtration}

The one-sided Witt algebra is a $\ZZ$-graded Lie algebra. Furthermore, the graded pieces are all one-dimensional, and $\WW_1$ only contains elements of degrees $-1$ and higher, making it very easy to determine whether a subspace of $\WW_1$ has finite or infinite codimension. Unfortunately, $\L_D$ is not a graded Lie algebra in general, but at least it has a natural filtration by degrees, which we define next. This will allow us to use associated graded techniques to aid with the proof of Theorem~\ref{thm:main}.

\begin{notation}
    Recall that $\T_{\widetilde{D}}$ is the tangent sheaf of $\widetilde{D}$. We write
    $$\widetilde{D} \setminus D = \sum_{p\in \widetilde{D}\setminus D} p$$
    as a divisor in $\widetilde{D}$. Given a divisor $N$ in $\widetilde{D}$, we write the twist of the tangent bundles as $\T_{\widetilde{D}}(N) \coloneqq \T_{\widetilde{D}} \otimes_{\widetilde{D}} \O_{\widetilde{D}}(N)$. For $n \in \ZZ$, define
    \begin{align*}
        \L_D^{\leq n} &\coloneqq H^0\left(\widetilde{D},\T_{\widetilde{D}}\big((n - 1)(\widetilde{D} \setminus D)\big)\right) = \{v \in \L_D \mid \deg_p(v) \leq n \text{ for all } p \in \widetilde{D} \setminus D\}, \\
        \g^{\leq n} &\coloneqq \g \cap \L_D^{\leq n} = \{v \in \g \mid \deg_p(v) \leq n \text{ for all } p \in \widetilde{D} \setminus D\}.
    \end{align*}
    It is easy to see that this defines an ascending filtration. Indeed, we have
    \begin{itemize}
        \item $\L_D^{\leq n} \subseteq \L_D^{\leq n + 1}$ for all $n \in \ZZ$.
        \item $\bigcup_{n \in \ZZ} \L_D^{\leq n} = \L_D$.
    \end{itemize}
    The associated graded algebras are denoted $\gr \L_D$ and $\gr \g$, whose graded components are
    $$\gr_n \L_D \coloneqq \frac{\L_D^{\leq n}}{\L_D^{\leq n - 1}}, \qquad \gr_n \g \coloneqq \frac{\g^{\leq n}}{\g^{\leq n - 1}}.$$
\end{notation}

\begin{remark}\label{rem:filtration is compatible with Lie structure}
    Although we will not need this for the proof of Theorem~\ref{thm:main}, the filtration above is compatible with the Lie algebra structure of $\L_D$. In other words, we have $[\L_D^{\leq n}, \L_D^{\leq m}] \subseteq \L_D^{\leq n + m}$ for all $n,m \in \ZZ$, which follows by Lemma~\ref{lem:degrees form a semigroup}.
\end{remark}

As we show next, this filtration is bounded below.

\begin{lemma}
    The filtration $\L_D^{\leq n}$ is bounded below. In other words, $\L_D^{\leq n} = \{0\}$ for all $n \ll 0$.
\end{lemma}
\begin{proof}
    Let $g$ be the genus of $\widetilde{D}$. Recall that, by definition, we have
    $$\L_D^{\leq n} = H^0\left(\widetilde{D},\T_{\widetilde{D}}\big((n-1)(\widetilde{D} \setminus D)\big)\right),$$
    which is the space of global sections of a line bundle of degree $2 - 2g + (n - 1)|\widetilde{D}\setminus D|$. If the degree of this line bundle is negative (in other words, if $n \ll 0$), then $\L_D^{\leq n} = 0$.
\end{proof}

Having a filtration that is bounded below is very useful for codimension arguments, thanks to the following result.

\begin{lemma}\label{lem:finite codimension filtration}
    Let $V$ be a filtered vector space with $V_m = 0$ for some $m \in \ZZ$, so that
    $$0 = V_m \subseteq V_{m + 1} \subseteq V_{m + 2} \subseteq \dots,$$
    with $\dim V_k < \infty$ for all $k \in \ZZ$ and $\bigcup_{k \in \ZZ} V_k = V$. Let $U$ be a subspace of $V$ and define $U_n \coloneqq U \cap V_n$ for all $n \in \ZZ$. Then $U$ has finite codimension in $V$ if and only if the map $\gr_n U \to \gr_n V$ is an isomorphism for $n \gg 0$, where $\gr_n V \coloneqq V_n/V_{n - 1}$ and $\gr_n U \coloneqq U_n/U_{n - 1}$.
\end{lemma}
\begin{proof}
   Define $W \coloneqq V/U$. Then $W$ is a filtered vector space by setting $W_n \coloneqq \frac{V_n + U}{U} \cong V_n/U_n$. Setting $\gr_n W \coloneqq W_n/W_{n - 1}$ and identifying $\gr_n U$ with $\frac{U_n + V_{n - 1}}{V_{n - 1}}$ (its image in $\gr_n V$), it is easy to see that
    $$\gr_n W \cong \frac{V_n}{U_n + V_{n - 1}} \cong \frac{\gr_n V}{\gr_n U}.$$
    Since $V_m = \{0\}$, we have $W_m = \{0\}$. Furthermore, we can use the assumption $\dim(V_n) < \infty$ to deduce that $W_n$ is finite-dimensional for all $n \in \ZZ$.

    Suppose the map $\gr_n U \to \gr_n V$ is an isomorphism for $n \gg 0$. In other words, there exists $N \in \NN$ such that $\gr_n U = \gr_n V$ for all $n \geq N$ (again, we are identifying $\gr_n U$ with $\frac{U_n + V_{n - 1}}{V_{n - 1}}$). Therefore, if $n \geq N$, then $\gr_n W = \{0\}$, from which it follows that $W_n = W_N$ for all $n \geq N$. Now the property $\bigcup_{n \in \ZZ} V_n = V$ implies that $W_N = \bigcup_{n \in \ZZ} W_n = W$, so we conclude that $W$ is finite-dimensional.
    
    For the converse, suppose $W$ is finite-dimensional. Let $\{w_1,\dots, w_d\}$ be a basis for $W$, and let
    $$N \coloneqq \min\{n \in \ZZ \mid w_i \in W_n \text{ for all } i\}.$$
    It follows that $W = W_n$ for all $n \geq N$. Therefore, we have
    $$\frac{\gr_n V}{\gr_n U} \cong \gr_n W = \frac{W_n}{W_{n - 1}} = \{0\}$$
    for all $n \geq N + 1$. In other words, the map $\gr_n U \to \gr_n V$ is an isomorphism for all $n \gg 0$. This concludes the proof.
\end{proof}

It follows easily from Lemma~\ref{lem:properties of deg(g)} that the dimension of the graded components of $\gr \L_D$ is at most $|\widetilde{D} \setminus D|$. In fact, as we show next, the dimension of $\gr_n \L_D$ is exactly $|\widetilde{D} \setminus D|$, provided $n$ is large enough, which is an easy consequence of Serre duality (or the Riemann--Roch theorem -- see, for example, \cite[Theorem II.5.4]{Silverman} or \cite[Section 4.1]{Schlichenmaier} for details).

\begin{proposition}\label{prop:dimension of graded pieces}
    For $n \gg 0$, we have $\dim(\gr_n \L_D) = |\widetilde{D} \setminus D|$. More precisely, a basis for $\gr_n \L_D$ (for $n \gg 0$) is given by $\{v_n^p + \L_D^{\leq n - 1} \mid p \in \widetilde{D} \setminus D\}$, where the elements $v_n^p \in \L_D$ are such that $\deg_p(v_n^p) = n$ and $\deg_q(v_n^p) < n$ for all $q \in \widetilde{D} \setminus (D \cup \{p\})$.
\end{proposition}
\begin{proof}
    We have an exact sequence of sheaves:
    $$0 \to \T_{\widetilde{D}}\big((n - 1)(\widetilde{D} \setminus D)\big) \to \T_{\widetilde{D}}\big(n(\widetilde{D} \setminus D)\big) \to \O_{\widetilde{D}\setminus D} \to 0.$$
    By Serre duality we have that $H^1\left(\widetilde{D},\T_{\widetilde{D}}\big((n - 1)(\widetilde{D} \setminus D)\big)\right)$ is dual to $$H^0\left(\widetilde{D}, \Omega_{\widetilde{D}}^1 \otimes \Omega^1_{\widetilde{D}} \big((1 - n)(\widetilde{D} \setminus D)\big)\right),$$ which is global sections of a line bundle of degree $2(2g-2)+(1-n)|\widetilde{D}\setminus D|.$  When this degree is negative (in other words, when $n \gg 0$), we have a short exact sequence of global sections
    $$ 0 \to \L^{\leq n}_D \to \L^{\leq n+1}_D \to \kk^{\widetilde{D}\setminus D} \to 0.$$
    This proves that $\dim(\gr_n \L_D) = |\widetilde{D} \setminus D|$ for all $n \gg 0$.

    We can similarly use Serre duality to show that the elements $v_n^p$ exist for all $p \in \widetilde{D} \setminus D$ and $n \gg 0$. It is easy to see that the set $\{v_n^p + \L_D^{\leq n - 1} \mid p \in \widetilde{D} \setminus D\}$ is linearly independent, by Lemma~\ref{lem:properties of deg(g)}\eqref{item:maximum condition}, and is therefore a basis for $\gr_n \L_D$.
\end{proof}

Thanks to Proposition~\ref{prop:dimension of graded pieces}, all we have to do to prove that the map $\gr_n \g \hookrightarrow \gr_n \L_D$ is an isomorphism is showing that $\dim(\gr_n \g) = |\widetilde{D} \setminus D|$ for all $n \gg 0$.

\section{Proof of Theorem~\ref{thm:main}}\label{sec:proof of main theorem}

In this section, we complete the proof of Theorem~\ref{thm:main}. The proof will follow by induction on $|\widetilde{D} \setminus D|$, the number of poles of $\g$ in $\widetilde{D}$.

\subsection{One-point case}

We now prove Theorem~\ref{thm:main} in the case $|\widetilde{D} \setminus D| = 1$, which is the base case of the induction. The proof is easy, since Proposition~\ref{prop:dimension of graded pieces} implies that $\dim(\gr_n \L_D) = 1$ for $n \gg 0$, so the only thing we need to do is show that $\gr_n \g \neq 0$ for $n \gg 0$.

\begin{proposition}\label{prop:one point with poles}
    Suppose $|\widetilde{D} \setminus D| = 1$. Then $\g$ has finite codimension in $\L_D$.
\end{proposition}
\begin{proof}
    By Proposition~\ref{prop:g has almost all poles}, we have $d_p(\g) = 1$. Then Lemma~\ref{lem:degrees form a semigroup} implies that there exist $w_n \in \g$ with $\deg_p(w_n) = n$ for all $n \gg 0$. It now follows by Proposition~\ref{prop:dimension of graded pieces} that the inclusion map $\gr \g \hookrightarrow \gr \L_D$ is an isomorphism in high degree, since $\dim(\gr_n \L_D) = 1$ for $n \gg 0$. Therefore, $\g$ has finite codimension in $\L_D$, by Lemma~\ref{lem:finite codimension filtration}.
\end{proof}

\subsection{Two or more points}

We proceed with the inductive step of Theorem~\ref{thm:main}. As a result, we now assume that $|\widetilde{D} \setminus D| \geq 2$ and explicitly state the inductive hypothesis.

\begin{IH}
    If $\h$ is a Lie subalgebra of $\L_C$ such that $|\widetilde{D}_\h \setminus D_\h| < |\widetilde{D} \setminus D|$, then $\h$ has finite codimension in $\L_{D_\h}$.
\end{IH}

It will be convenient to establish notation for the set of possible $p$-degrees of elements of $\g$ (where $p \in \widetilde{D}$).

\begin{notation}
    Given $p \in \widetilde{D}$, define the following subset of $\ZZ$:
    $$\deg_p(\g) \coloneqq \{\deg_p(w) \mid w \in \g\}.$$
\end{notation}

First, we prove that $\deg_p(\g)$ contains all large enough integers, provided $p \in \widetilde{D} \setminus D$. This means that $\g$ contains vector fields with poles of all large enough orders at $p$.

\begin{lemma}\label{lem:poles can be arbitrarily large}
    Let $p \in \widetilde{D} \setminus D$. Then there exists $N \in \ZZ$ such that
    $$\ZZ_{\geq N} \subseteq \deg_p(\g).$$
\end{lemma}
\begin{proof}
    We claim that $\deg_p(\g)$ is unbounded above. By definition, we have $D = \{q \in \widetilde{D} \mid \deg_q(w) \leq 1 \text{ for all } w \in \g\}$. Since $p \notin D$, there exists $w \in \g$ such that $\deg_p(w) \geq 2$. By Proposition~\ref{prop:g has almost all poles}, there exists $v \in \g$ such that $\deg_p(v)$ is coprime to $\deg_p(w)$. Now Lemma~\ref{lem:degrees form a semigroup} implies that
    $$\deg_p(\ad_w^n(v)) = n \deg_p(w) + \deg_p(v) \geq 2N + \deg_p(v)$$
    for all $n \in \NN$. Since $\ad_w^n(v) \in \g$, we conclude that $\deg_p(\g)$ is unbounded above, as claimed.

    Now, Proposition~\ref{prop:g has almost all poles} implies that $d_p(\g) = 1$, so the result now follows by Corollary~\ref{cor:gcd of degrees}.
\end{proof}

Next, we show that the set
$$\{(\deg_p(w), \deg_q(w)) \mid w \in \g\}$$
does not lie on the diagonal of $\ZZ \times \ZZ$. One can interpret this result to mean that situations such as the one in Example \ref{ex:harder example} are excluded.

\begin{lemma}\label{lem:degrees at different points are different}
    Let $p,q \in \widetilde{D} \setminus D$ with $p \neq q$. Then there exists $w \in \g$ such that $\deg_p(w) \neq \deg_q(w)$.
\end{lemma}
\begin{proof}
    Assume, for a contradiction, that $\deg_p(w) = \deg_q(w)$ for all $w \in \g$. Fix an element $\del \in \Q_D \nonzero$ and let $f\del, g\del \in \g$, where $f, g \in \kk(D)$. Then
    $$\ord_p\left(\frac{f}{g}\right) = \ord_p(f) - \ord_p(g) = \ord_q(f) - \ord_q(g) = \ord_q\left(\frac{f}{g}\right).$$
    Therefore, all elements of $R(\g)$ have equal orders at $p$ and $q$. It follows that all elements of $F(\g) = \kk(D)$ have equal orders at $p$ and $q$, a contradiction.
\end{proof}

We now set notation for the subalgebra which will allow us to carry out the induction proof. This will be the subalgebra of $\g$ obtained by ``filling in'' one of the points $p \in \widetilde{D} \setminus D$.

\begin{notation}
    Given $p \in \widetilde{D} \setminus D$, we define the following subalgebra of $\g$:
    $$\g_p \coloneqq \g \cap \L_{D \cup \{p\}} = \{w \in \g \mid \ord_p(w) \geq 0\} = \{w \in \g \mid \deg_p(w) \leq 1\}.$$
    We let $\widetilde{D}_p \coloneqq \widetilde{D}_{\g_p}$ and $D_p \coloneqq D_{\g_p}$ be the curves associated to $\g_p$.
    
    We further define $\g_p' \coloneqq \{w \in \g \mid \deg_p(w) \leq -2\} \subseteq \g_p$.
\end{notation}

The reason for introducing the subalgebra $\g_p'$ is that having an element $w \in \g_p$ with $\deg_p(w) \leq -2$ allows us to decrease the $p$-degree of other vector fields by taking repeated Lie brackets with $w$, by Lemma~\ref{lem:degrees form a semigroup}. This will be useful later on.

Our next goal is to prove that the curve $D_p$ really is obtained from $D$ by ``filling in'' $p$. In other words, we want to prove that $D_p = D \cup \{p\}$ (for some appropriate choice of point $p \in \widetilde{D} \setminus D$), which will allow us to use the inductive hypothesis on $\g_p$. The main challenge is proving that $\g_p$ has poles at every point in $\widetilde{D} \setminus D$ other than $p$ itself. The following result one step in proving this.

\begin{proposition}\label{prop:dichotomy of poles at p and q}
    Let $p,q \in \widetilde{D} \setminus D$ with $p \neq q$. Then at least one of the following holds:
    \begin{enumerate}
        \item There exists $w \in \g_p'$ with a pole at $q$.
        \item There exists $w \in \g_q'$ with a pole at $p$.
    \end{enumerate}
\end{proposition}
\begin{proof}
    Assume, for a contradiction, that no elements of $\g_p'$ and $\g_q'$ have poles at $p$ or $q$. Given $n \in \deg_p(\g)$, define
    $$S_n \coloneqq \{\deg_q(w) \mid w \in \g \text{ and } \deg_p(w) = n\}.$$
    Suppose $S_n$ is unbounded below for some $n \in \deg_p(\g)$. Let $v \in \g$ such that $\deg_p(v) \geq 2 - n$ (which exists by Lemma~\ref{lem:poles can be arbitrarily large}), and let $u \in \g$ such that $\deg_p(u) = n$ and $\deg_q(u) \leq -2 - \deg_q(v)$ (which exists by the assumption that $S_n$ is unbounded below). Then
    $$\deg_p([u,v]) = \deg_p(u) + \deg_p(v) \geq 2, \qquad \deg_q([u,v]) = \deg_q(u) + \deg_q(v) \leq -2,$$
    so that $w \coloneqq [u,v] \in \g_q'$ with $\deg_p(w) \geq 2$ (so $w$ has a pole at $p$), a contradiction. It follows that $S_n$ is bounded below for all $n \in \deg_p(\g)$.

    Define a function
    \begin{align*}
        f \colon \deg_p(\g) &\to \deg_q(\g) \\
        n &\mapsto \min(S_n).
    \end{align*}
    Given $n \in \deg_p(\g)$, we let $w_n \in \g$ such that $\deg_p(w_n) = n$ and $\deg_q(w_n) = \min(S_n)$. We will prove that $f$ is the identity map through a series of claims, at which point we will be able to use Lemma~\ref{lem:degrees at different points are different} to prove the result.

    \begin{claim}\label{claim:f is injective}
        $f$ is injective.
    \end{claim}
    \begin{proof}[Proof of Claim \ref{claim:f is injective}]
        Assume, for a contradiction, that $f(n) = f(m)$ for some $n > m$. Then $\deg_q(w_n) = \deg_q(w_m)$, so Lemma~\ref{lem:properties of deg(g)} implies that there exists $v \in \spn\{w_n,w_m\}$ with
        $$\deg_p(v) = n, \qquad \deg_q(v) < \deg_q(w_n) = f(n).$$
        This contradicts the definition of $f$.
    \end{proof}

    \begin{claim}\label{claim:f is sub-additive}
        $f(n + m) \leq f(n) + f(m)$ for all $n,m \in \deg_p(\g)$ with $n \neq m$.
    \end{claim}
    \begin{proof}[Proof of Claim \ref{claim:f is sub-additive}]
        For $n,m \in \deg_p(\g)$ with $n \neq m$, we have
        \begin{align*}
            \deg_p([w_n,w_m]) &= \deg_p(w_n) + \deg_p(w_m) = n + m, \\
            \deg_q([w_n,w_m]) &= \deg_q(w_n) + \deg_q(w_m) = f(n) + f(m),
        \end{align*}
        where we used that $\deg_q(w_n) \neq \deg_q(w_m)$ by Claim \ref{claim:f is injective}. It follows that $f(n + m) \leq f(n) + f(m)$ by definition of $f$.
    \end{proof}

    Using Lemma~\ref{lem:poles can be arbitrarily large}, choose $N \geq 2$ such that $\ZZ_{\geq N} \subseteq \deg_p(\g)$.

    \begin{claim}\label{claim:f is increasing}
        $f(n) \geq n$ for all $n \in \ZZ_{\geq N}$.
    \end{claim}
    \begin{proof}[Proof of Claim \ref{claim:f is increasing}]
        Assume, for a contradiction, that $f(n) < n$ for some $n \in \ZZ_{\geq N}$. For $t \in \ZZ_{\geq 2}$, define
        $$X_t \coloneqq \{b n + r \mid 2 \leq b \leq t, 1 \leq r \leq \min(n,b - 1)\}.$$
        For $2 \leq b \leq n + 1$, there are $b - 1$ integers $m$ that can be written as $m = b n + r$ for some $1 \leq r < b$. Therefore,
        $$|X_{n + 1}| = 1 + 2 + \dots + n = \frac{n(n + 1)}{2}.$$
        For $b \geq n + 2$, there are $n$ integers $m$ that can be written as $m = b n + r$ for some $1 \leq r \leq n$, and thus
        $$|X_t| = \frac{n(n + 1)}{2} + (t - (n + 1))n = n t - \frac{n(n + 1)}{2}$$
        for all $t \in \ZZ_{\geq n + 1}$.
        
        Let $m = b n + r \in X_t$ for some $t \in \ZZ_{\geq 2}$, where $2 \leq b \leq t$ and $1 \leq r \leq \min(n, b - 1)$. Notice that $m$ can be written as
        $$m = k n + r(n + 1),$$
        where $k \coloneqq b - r > 0$. Applying Claim \ref{claim:f is sub-additive} inductively, we have
        $$f(m) = f(b n + r) = f(k n + r (n + 1)) \leq k f(n) + r f(n + 1).$$
        Letting $c \coloneqq f(n + 1)$ and $d \coloneqq f(n) - f(n + 1)$ (so that $f(n) = c + d$), we deduce that
        $$f(m) \leq k(c + d) + rc = (k + r)c + kd = b c + kd.$$
        Define
        $$a \coloneqq \begin{cases}
            d, &\text{if } d > 0, \\
            nd, &\text{if } d < 0.
        \end{cases}$$
        Since $1 \leq r \leq n$ and $k = b - r$ by definition, we see that
        $$b - n \leq k \leq b - 1.$$
        If $d > 0$, then $k \leq b - 1$ implies that $kd \leq (b - 1)d$, so we deduce from the above that
        $$f(m) \leq bc + kd \leq bc + (b - 1)d = (c + d)b - d = f(n)b - a \leq (n - 1)b - a,$$
        since $f(n) \leq n - 1$ by assumption. If $d < 0$, then $k \geq b - n$ implies that $kd \leq (b - n)d$, so we similarly deduce that
        $$f(m) \leq bc + kd \leq bc + (b - n)d = (c + d)b - nd = f(n)b - a \leq (n - 1)b - a,$$
        where we once again used that $f(n) \leq n - 1$. In either case, we have $f(m) \leq (n - 1)b - a$. Therefore, we see that
        $$f(m) \leq (n - 1)t - a$$
        for all $t \in \ZZ_{\geq 2}$ and $m \in X_t$.

        Notice that $\deg_q(w_n) \geq -1$ for all $n \in \deg_p(\g) \cap \ZZ_{\geq 2}$, since no element of $\g_q'$ can have a pole at $p$ by assumption. In particular, this means that $f(\ZZ_{\geq N}) \subseteq \ZZ_{\geq -1}$. Define
        $$Y_t \coloneqq \{k \in \ZZ \mid -1 \leq k \leq (n - 1)t - a\}.$$
        Let $t$ be a large enough integer such that
        $$nt - \frac{n(n + 1)}{2} > (n - 1)t - a + 2.$$
        Then the restriction
        $$\restr{f}{X_t} \colon X_t \to Y_t$$
        is an injective function. This is a contradiction, because
        $$|X_t| = nt - \frac{n(n + 1)}{2} > (n - 1)t - a + 2 = |Y_t|.$$
        It follows that $f(n) \geq n$ for all $n \in \ZZ_{\geq N}$.
    \end{proof}

    \begin{claim}\label{claim:f is the identity}
        $f(n) = n$ for all $n \in \deg_p(\g)$.
    \end{claim}
    \begin{proof}[Proof of Claim \ref{claim:f is the identity}]
        Let $g \colon \deg_q(\g) \to \deg_p(\g)$ be the function one gets by swapping the roles of $p$ and $q$ above. Then Claims \ref{claim:f is injective}--\ref{claim:f is increasing} are true if we replace $f$ by $g$. More precisely, use Lemma~\ref{lem:poles can be arbitrarily large} to get an integer $N' \geq 2$ such that $\ZZ_{\geq N'} \subseteq \deg_q(\g)$. Then $g(n) \geq n$ for all $n \in \ZZ_{\geq N'}$, by Claim \ref{claim:f is increasing} applied to $g$.

        Recalling that $\deg_q(w_n) = f(n)$ and $\deg_p(w_n) = n$, it follows that
        \begin{equation}\label{eq:g(f(n)) is less than n}
            g(f(n)) = \min\{\deg_p(w) \mid w \in \g \text{ and } \deg_q(w) = f(n)\} \leq n
        \end{equation}
        for all $n \in \deg_p(\g)$. Let $M \coloneqq \max(N,N')$. Given $n \in \ZZ_{\geq M}$, we have $f(n) \geq n \geq M$, and thus
        $$n \geq g(f(n)) \geq f(n) \geq n$$
        for all $n \in \ZZ_{\geq M}$, where the first inequality follows from \eqref{eq:g(f(n)) is less than n}, the second and third ones follow from Claim \ref{claim:f is increasing} applied to $g$ and $f$, respectively. It follows that $f(n) = n$ for all $n \in \ZZ_{\geq M}$.

        Now $n \in \deg_p(\g)$, and let $m \in \deg_p(\g)$ such that $m \geq M$ and $n + m \geq M$. Applying the above, we get $f(n + m) = n + m$ and $f(m) = m$, and thus
        $$n + m = f(n + m) \leq f(n) + f(m) = f(n) + m,$$
        where we used Claim \ref{claim:f is sub-additive} to deduce that $f(n + m) \leq f(n) + f(m)$. It follows that $f(n) \geq n$ for all $n \in \deg_p(\g)$. Similarly, we also get that $g(n) \geq n$ for all $n \in \deg_q(\g)$. We conclude that
        $$n \geq g(f(n)) \geq f(n) \geq n$$
        for all $n \in \deg_p(\g)$, from which it follows that $f(n) = n$ for all $n \in \deg_p(\g)$.
    \end{proof}

    By the definition of $f$, Claim \ref{claim:f is the identity} implies that $\deg_p(w) \leq \deg_q(w)$ for all $w \in \g$. Swapping the roles of $p$ and $q$, we also get that $\deg_p(w) \geq \deg_q(w)$ for all $w \in \g$. In other words, we have $\deg_p(w) = \deg_q(w)$ for all $w \in \g$. This contradicts Lemma~\ref{lem:degrees at different points are different}, which concludes the proof.
\end{proof}

To prove that $D_p = D \cup \{p\}$, we need to show that $\g_p$ has poles at every point in $\widetilde{D} \setminus D$ other than $p$ (otherwise $D_p$ might be bigger). In other words, we must show that there exist vector fields in $\g$ with poles at every point in $\widetilde{D} \setminus (D \cup \{p\})$, but no pole at $p$. We now show that there is at least one choice of $p \in \widetilde{D} \setminus D$ where this is the case.

\begin{lemma}\label{lem:good point}
    There exist $p \in \widetilde{D} \setminus D$ and $w \in \g$ such that
    $$\ord_p(w) \geq 0, \qquad \ord_q(w) < 0$$
    for all $q \in \widetilde{D} \setminus D$ with $q \neq p$. In other words, there is an element of $\g$ with poles at every point in $\widetilde{D} \setminus D$ except for $p$.
\end{lemma}
\begin{proof}
    For every $p,q \in \widetilde{D} \setminus D$ with $p \neq q$, we can use Proposition~\ref{prop:dichotomy of poles at p and q} to choose $w_{pq} \in \g$ such that $\ord_{p}(w_{pq})$ and $\ord_{q}(w_{pq})$ have different signs (where $w_{pq} = w_{qp}$). Call a point $p \in \widetilde{D} \setminus D$ \emph{good} if for all $q \in \widetilde{D} \setminus D$ with $q \neq p$, there exist $r,s \in \widetilde{D} \setminus D$ with $r \neq s$ such that $\ord_{p}(w_{rs}) \geq 0$ and $\ord_{q}(w_{rs}) < 0$. Otherwise, $p$ is \emph{bad}.

    We claim that there is at least one good point in $\widetilde{D} \setminus D$. For $p \in \widetilde{D} \setminus D$, define
    $$m_p \coloneqq \left.\left|\Big\{\{r,s\} \subseteq \widetilde{D} \setminus D \ \right| r \neq s \text{ and } \ord_p(w_{rs}) < 0\Big\}\right|.$$
    Suppose $p$ is bad. Then there must exist $q \in \widetilde{D} \setminus D$ with $q \neq p$ such that $\ord_{p}(w_{rs}) \geq 0$ implies that $\ord_{q}(w_{rs}) \geq 0$ for all $r,s \in \widetilde{D} \setminus D$. Since $\ord_{p}(w_{pq})$ and $\ord_{q}(w_{pq})$ have different signs, it follows that $\ord_{p}(w_{pq}) < 0$ and $\ord_{q}(w_{pq}) \geq 0$. In particular, we see that $m_p > m_q$. Therefore, if we choose $p \in \widetilde{D} \setminus D$ such that $m_p$ is minimal, then $p$ must be good.

    Fix a good point $p \in \widetilde{D}$, and let $u,v \in \L_D$. By Lemma~\ref{lem:properties of deg(g)}\eqref{item:maximum condition}, if $\ord_{r}(u) < 0$ and $\ord_{s}(v) < 0$, then there exists $w \in \spn\{u,v\}$ with $\ord_{r}(w) < 0$ and $\ord_{s}(w) < 0$. Therefore, since $p$ is good, there exists $w \in \spn\{w_{ij} \mid i \neq j\}$ such that $\ord_p(w) \geq 0$ and $\ord_{q}(w) < 0$ for all $q \in \widetilde{D} \setminus D$ with $q \neq p$. This concludes the proof.
\end{proof}

For the rest of this section, we fix a point $p \in \widetilde{D} \setminus D$ as in Lemma~\ref{lem:good point}. In other words, $p$ is a point such that the Lie algebra $\g_p$ has poles at every point in $\widetilde{D} \setminus (D \cup \{p\})$. We now show that $\g_p$ is infinite-dimensional.

\begin{lemma}\label{lem:g_p is infinite-dimensional}
    The set $\deg_p(\g)$ is unbounded below. Consequently, $\g_p$ is infinite-dimensional.
\end{lemma}
\begin{proof}
    Keep the notation $w_{pq} \in \g$ and $m_p$ from the proof of Lemma~\ref{lem:good point}. Note that Proposition~\ref{prop:dichotomy of poles at p and q} allows us to choose each $w_{pq}$ so that $\deg_p(w_{pq}) \geq 3$ or $\deg_q(w_{pq}) \geq 3$ for every $q \in \widetilde{D} \setminus (D \cup \{p\})$.
    
    The point $p$ is constructed in the proof of Lemma~\ref{lem:good point} by choosing $m_p$ to be minimal. This means that there exists at least one point $q \in \widetilde{D} \setminus D$ with $p \neq q$ such that $\ord_p(w_{pq}) \geq 3$, meaning $\deg_p(w_{pq}) \leq -2$. Let $w \coloneqq w_{pq}$, and choose $v \in \g$ so that $\deg_p(v)$ is coprime to $\deg_p(w)$ (which exists by Lemma~\ref{lem:poles can be arbitrarily large}). Now, Lemma~\ref{lem:degrees form a semigroup} gives
    $$\deg_p(\ad_w^n(v)) = n\deg_p(w) + \deg_p(v) \leq \deg_p(v) - 2n$$
    for all $n \in \NN$, since $\deg_p(w) \leq -2$. This proves that $\deg_p(\g)$ is unbounded below.
    
    For the final sentence of the statement, the fact that $\deg_p(\g)$ is unbounded below gives infinitely many linearly independent elements of $\g_p$.
\end{proof}

Next, we prove that $\g_p$ has the same field of ratios as $\g$, meaning we do not lose too much when passing from $\g$ to $\g_p$.

\begin{proposition}\label{prop:field of g_p}
    We have $F(\g_p) = \kk(D)$.
\end{proposition}
\begin{proof}
    Let $\pi \colon \widetilde{D} \to \widetilde{D}_p$ be the dominant morphism induced by the inclusion $\kk(\widetilde{D}_p) \hookrightarrow \kk(\widetilde{D})$. Given a point $q \in \widetilde{D}$, we let $e_q$ be the ramification index of $\pi$ at $q$.

    Let $q \in \widetilde{D} \setminus (D \cup \{p\})$. Letting $x \coloneqq \pi(q) \in \widetilde{D}_p$, it follows from Lemma~\ref{lem:ramification} that
    $$\deg_q(v) = e_q \deg_x(v)$$
    for all $v \in \g_p$. Therefore, $\deg_q(v)$ is a multiple of $e_q$ for all $v \in \g_p$.

    Assume, for a contradiction, that $e_q \geq 2$. Choose $u \in \g$ with $\deg_p(u) \leq -1$, which exists by Lemma~\ref{lem:g_p is infinite-dimensional}. Note that $\deg_q(u)$ is a multiple of $e_q$ by the above, since $\deg_p(u) \leq -1$ implies that $u \in \g_p$. Let $v \in \g$ such that $\deg_q(v)$ is coprime to $e_q$, which exists by Lemma~\ref{lem:poles can be arbitrarily large}. Notice that $\deg_p(v) \geq 2$, since an element whose $p$-degree is less than or equal to $1$ is necessarily an element of $\g_p$. Therefore, we may choose $v$ so that it has minimal $p$-degree among the elements of $\g$ whose $q$-degree is coprime to $e_q$. Then $[u,v] \in \g$, and
    $$\deg_q([u,v]) = \deg_q(u) + \deg_q(v)$$
    is coprime to $e_q$. However, we also have
    $$\deg_p([u,v]) = \deg_p(u) + \deg_p(v) < \deg_p(v),$$
    which contradicts the minimality of $\deg_p(v)$. Therefore, we must have $e_q = 1$, in other words, $\pi$ is unramified at $q$. Notice that this holds for all points in $\widetilde{D} \setminus (D \cup \{p\})$, since $q$ was chosen arbitrarily. In other words, $\pi$ is unramified at all points of $\widetilde{D}$ where $\g_p$ has a pole.

    We claim that $\pi^{-1}(x) = \pi^{-1}(\pi(q)) = \{q\}$. Assume, for a contradiction, that $\pi(r) = \pi(q)$ for some $r \in \widetilde{D} \setminus \{q\}$. It follows by the above that
    $$\deg_r(v) = e_r\deg_x(v) = e_r\deg_q(v)$$
    for all $v \in \g_p$. By Lemma~\ref{lem:good point}, there exists $w \in \g_p$ such that $\deg_q(w) \geq 2$. This implies that
    $$\deg_r(w) = e_r \deg_q(w) \geq 2e_r \geq 2,$$
    so $w$ has a pole at $r$, and thus $r \in \widetilde{D} \setminus (D \cup \{p,q\})$. If $|\widetilde{D} \setminus D| = 2$, then $\widetilde{D} = D \cup \{p,q\}$, so this is a contradiction.
    
    Otherwise, if $|\widetilde{D} \setminus D| \geq 3$, we need to work a bit harder. Since $\pi$ is unramified at the poles of $\g_p$, it follows that $e_r = 1$, so $\deg_q(v) = \deg_r(v)$ for all $v \in \g_p$. Using Lemma~\ref{lem:degrees at different points are different}, choose $v \in \g$ with $\deg_q(v) \neq \deg_r(v)$, and let $u \in \g_p$ such that $\deg_p(u) \leq 1 - \deg_p(v)$ (which exists by Lemma~\ref{lem:g_p is infinite-dimensional}). Then $[u,v] \in \g$ and
    $$\deg_p([u,v]) = \deg_p(u) + \deg_p(v) \leq 1,$$
    so $[u,v] \in \g_p$. Furthermore, since $u \in \g_p$, it follows from the above that $\deg_q(u) = \deg_r(u)$, and thus
    $$\deg_q([u,v]) = \deg_q(u) + \deg_q(v) \neq \deg_r(u) + \deg_r(v) = \deg_r([u,v]).$$
    Therefore, $[u,v]$ is an element of $\g_p$ whose $q$-degree is not equal to its $r$-degree, which is a contradiction. It follows that $\pi^{-1}(\pi(q)) = \{q\}$, as claimed.

    Now, by \cite[Proposition II.2.6(a)]{Silverman}, we have
    $$\deg(\pi) = \sum_{r \in \pi^{-1}(x)} e_r = e_q = 1,$$
    since $\pi^{-1}(x) = \{q\}$ and $e_q = 1$. Therefore, $\pi$ is an isomorphism, meaning the inclusion $\kk(\widetilde{D}_p) \subseteq \kk(\widetilde{D})$ is in fact an equality. In other words, $F(\g_p) = \kk(\widetilde{D}_p) = \kk(\widetilde{D}) = \kk(D)$, which concludes the proof.
\end{proof}

Finally, we now have all the necessary ingredients to prove that $D_p = D \cup \{p\}$. This is the final result we need to prove Theorem~\ref{thm:main}.

\begin{corollary}\label{cor:g_p has poles everywhere}
    We have $D_p = D \cup \{p\}$. Consequently, $\g_p$ has finite codimension in $\L_{D \cup \{p\}}$.
\end{corollary}
\begin{proof}
     Note that $\g_p$ is a subalgebra of $\L_{D \cup \{p\}}$ by definition. By Proposition~\ref{prop:field of g_p}, we know that $F(\g_p) = \kk(D)$, from which it follows that $\widetilde{D}_p = \widetilde{D}$. Since the element $w \in \g_p$ from Lemma~\ref{lem:good point} has poles at every point in $\widetilde{D} \setminus (D \cup \{p\})$, we conclude that $D_p = D \cup \{p\}$. The final sentence of the statement follows from the inductive hypothesis.
\end{proof}

We conclude this section by proving Theorem~\ref{thm:main}.

\begin{proof}[Proof of Theorem~\ref{thm:main}]
    As mentioned at the beginning of this section, we proceed by induction on $d \coloneqq |\widetilde{D} \setminus D|$. The base case $d = 1$ is Proposition~\ref{prop:one point with poles}.
    
    For the inductive step, assume $d \geq 2$ and suppose we have proved that any subalgebra $\h \subseteq \L_C$ with $|\widetilde{D}_\h \setminus D_\h| < d$ has finite codimension in $\L_{D_\h}$.

    To prove that $\g$ has finite codimension in $\L_D$, it suffices to prove that $\dim(\gr_n \g) = d$ for all $n \gg 0$, by Lemma~\ref{lem:finite codimension filtration} and Proposition~\ref{prop:dimension of graded pieces}. By Corollary~\ref{cor:g_p has poles everywhere}, the subalgebra $\g_p$ has finite codimension in $\L_{D \cup \{p\}}$, where $p \in \widetilde{D} \setminus D$ is a point satisfying the properties of Lemma~\ref{lem:good point}. It follows by Lemma~\ref{lem:finite codimension filtration} that there exists $k \in \ZZ$ such that $\gr_n \g_p \cong \gr_n \L_{D \cup \{p\}}$ for all $n \geq k$. Proposition~\ref{prop:dimension of graded pieces} then implies that $\dim(\gr_n \g_p) = d - 1$ for all $n \geq k$. In fact, Proposition~\ref{prop:dimension of graded pieces} also gives a description of a basis of $\gr_n \g_p$ for $n \geq k$: there exist elements $v_n^q \in \g_p$ with the property that $\deg_q(v_n^q) = n$ and $\deg_r(v_n^q) < n$ for all $r \in \widetilde{D} \setminus \{p,q\}$. Then the set $\{v_n^q + \g_p^{\leq n - 1} \mid q \in \widetilde{D} \setminus (D \cup \{p\})\}$ is a basis for $\gr_n \g_p$. Notice that, since $v_n^q \in \g_p$, we have $\deg_p(v_n^q) \leq 1$. In particular, if $n \geq 2$ then we also have $\deg_p(v_n^q) < n$.
    
    Thanks to Proposition~\ref{prop:dimension of graded pieces}, to prove that $\dim(\gr_n \g) = d$ for $n \gg 0$ it suffices to construct an element $v_n^p \in \g$ such that $\deg_p(v_n^p) = n$ and $\deg_q(v_n^p) < n$ for all $q \in \widetilde{D} \setminus (D \cup \{p\})$. By Lemma~\ref{lem:poles can be arbitrarily large}, there exists $N \in \ZZ$ such that $\ZZ_{\geq N} \subseteq \deg_p(\g)$.
    
    Let $n \geq \max(N,k,2)$. By construction, there exists $w \in \g$ such that $\deg_p(w) = n$. Now, Lemma~\ref{lem:properties of deg(g)} implies that there exists
    $$v_n^p \in \spn\{w,v_m^q \mid q \in \widetilde{D} \setminus (D \cup \{p\}), m \geq n\}$$
    such that $\deg_p(v_n^p) = n$ and $\deg_q(v_n^p) < n$ for all $q \in \widetilde{D} \setminus (D \cup \{p\})$. It follows that the set $\{v_n^q + \g^{\leq n - 1} \mid q \in \widetilde{D} \setminus D\}$ is a basis for $\gr_n \g$, so $\dim(\gr_n \g) = d$. This concludes the proof.
\end{proof}

The rest of the paper is devoted to presenting some applications of Theorem~\ref{thm:main}.

\section{Dixmier property of Krichever--Novikov algebras}\label{sec:Dixmier}

In this section, we characterize precisely when a Krichever--Novikov algebra satisfies the \emph{Dixmier property}, defined below.

\begin{definition}\label{def:Dixmier}
    We say that an (associative or Lie) algebra $A$ \emph{satisfies the Dixmier property} if all injective endomorphisms of $A$ are automorphisms.
\end{definition}

\begin{example}\label{ex:Dixmier}
    The one-sided Witt algebra $\WW_1 = \L_{\Aa^1} = \Der(\kk[t])$ satisfies the Dixmier property. This was shown in \cite{Du} and \cite[Corollary 4.4]{BellBuzaglo}.

    On the other hand, the full Witt algebra $W = \L_{\Aa^1 \setminus \{0\}} = \Der(\kk[t,t^{-1}])$ does not satisfy the Dixmier property. This is most easily seen by the existence of the injective endomorphism
    \begin{align*}
        \varphi_k \colon W &\to W \\
        e_n &\mapsto \frac{1}{k} e_{kn}
    \end{align*}
    for all positive integers $k$.
\end{example}

In fact, the Witt algebra $W$ is the only Lie algebra of vector fields on a curve which does not satisfy the Dixmier property. We state this as a theorem below, and devote the rest of this section to proving it.

\begin{theorem}\label{thm:Dixmier}
    Let $C$ be a smooth affine curve. Then $\L_C$ satisfies the Dixmier property if and only if $C \not\cong \Aa^1 \nonzero$.
\end{theorem}

The reason that $W$ does not satisfy the Dixmier property while the other Krichever--Novikov algebras do is essentially the following geometric result, which proves that the only affine curve with an \'etale self-map is $\Aa^1 \nonzero$.

\begin{proposition}\label{prop:Jacobian curves}
    Let $C$ be a smooth affine curve. Then there exists an \'etale map $f \colon C \to C$ which is not an isomorphism if and only if $C \cong \Aa^1 \nonzero$.
\end{proposition}
\begin{proof}
    If $C = \Aa^1 \nonzero$, then the map $f \colon C \to C$ given by $x \mapsto x^2$ is \'etale and not an isomorphism.

    For the converse, suppose $f \colon C \to C$ is an \'etale map which is not an isomorphism. Let $\widetilde{C}$ be the projective compactification of $C$, let $g$ be the genus of $\widetilde{C}$, and let $d \coloneqq \deg(f) \geq 2$ (since $f$ is not an isomorphism). Then $f$ extends uniquely to a morphism $f \colon \widetilde{C} \to \widetilde{C}$. Now $f$ must be a dominant morphism, and thus is surjective. Let $X \coloneqq f^{-1}(\widetilde{C} \setminus C)$. Since $f(C) \subseteq C$, it follows that $X \subseteq \widetilde{C} \setminus C$. Now, $f$ restricts to a surjective map $\restr{f}{X} \colon X \to \widetilde{C} \setminus C$. Since the set $\widetilde{C} \setminus C$ is finite, it follows that $X = \widetilde{C} \setminus C$ and that $\restr{f}{X}$ is injective. In other words, $f$ is injective on $\widetilde{C} \setminus C$, and thus $f$ is totally ramified on $\widetilde{C} \setminus C$.
    
    By Hurwitz's formula \cite[Theorem II.5.9]{Silverman}, we have
    $$\sum_{p \in \widetilde{C}} (e_p - 1) = (1 - d)(2g - 2),$$
    where $e_p$ is the ramification index of $f$ at $p$. In particular, this means that
    $$(1 - d)(2g - 2) \geq 0.$$
    Since $d \geq 2$, this forces $g = 0$ or $g = 1$. Furthermore, $f$ is \'etale on $C$ (meaning $e_p = 1$ for all $p \in C$) and is totally ramified on $\widetilde{C} \setminus C$ (meaning $e_p = d$ for all $p \in \widetilde{C} \setminus C$). Letting $r \coloneqq |\widetilde{C} \setminus C|$, it follows that
    $$r(d - 1) = (1 - d)(2g - 2).$$
    Since $d \geq 2$, we conclude that $r = 2 - 2g$. If $g = 1$, this means that $r = 0$, which is impossible since $C$ is an affine curve, while if $g = 0$, it follows that $r = 2$. We conclude that the only possibility is that $C = \PP^1 \setminus \{2 \text{ points}\} \cong \Aa^1 \nonzero$.
\end{proof}

To prove Theorem~\ref{thm:Dixmier}, we will use some techniques developed in \cite{Grabowski} and \cite{Siebert} to show that a proper finite-codimensional subalgebra of $\L_C$ is not simple. Since $\L_C$ is a simple Lie algebra \cite{Siebert}, this will imply that $\L_C$ is not isomorphic to any of its subalgebras of finite codimension. We will then need Theorem~\ref{thm:main} to deal with the infinite codimension case. To use the techniques from \cite{Grabowski}, we require some notation.

\begin{notation}
    Given a $\kk$-algebra $A$ and an ideal $I \ideal A$, we write
    $$\Der_I(A) \coloneqq \{d \in \Der(A) \mid d(a) \in I \text{ for all } a \in A\}.$$
\end{notation}

It is clear that $\Der_I(A)$ is a Lie subalgebra of $\Der(A)$ which is also an $A$-submodule. We now state Grabowski's fundamental lemma from \cite{Grabowski}, which will be essential for the proof of Theorem~\ref{thm:Dixmier}.

\begin{lemma}[{\cite[Fundamental Lemma]{Grabowski}}]\label{lem:fundamental}
    Let $A$ be a $\kk$-algebra, and let $\g$ be a proper subalgebra of $\Der(A)$ of finite codimension. Then there exists a proper ideal $I_0$ of $A$ such that $\g \subseteq \Der_P(A)$ for all prime ideals $P \ideal A$ with $I_0 \subseteq P$.
\end{lemma}

It now follows easily from Grabowski's fundamental lemma that a proper subalgebra of $\L_C$ of finite codimension is not simple.

\begin{corollary}\label{cor:proper subalgebra not simple}
    Let $C$ be a smooth affine curve and let $\g$ be a proper Lie subalgebra of $\L_C$ of finite codimension. Then $\g$ is not simple.
\end{corollary}
\begin{proof}
    By definition, we have $\L_C = \Der(\O_C)$, so we can apply Grabowski's fundamental lemma (Lemma~\ref{lem:fundamental}) to deduce that there exists a maximal ideal $\mathfrak{m} \ideal \O_C$ such that $\g \subseteq \Der_{\mathfrak{m}}(\O_C)$. Letting $p \in C$ be the point corresponding to the maximal ideal $\mathfrak{m}$, write $\L_C(-p) \coloneqq \Der_{\mathfrak{m}}(\O_C)$. It is easy to see that
    $$\L_C(-p) = \{v \in \L_C \mid \ord_p(v) \geq 1\} = \{v \in \L_C \mid \deg_p(v) \leq 0\}.$$
    For a Lie algebra $L$, define $D(L) \coloneqq [L,L]$. By Lemma~\ref{lem:degrees form a semigroup}, it follows that $\deg_p([u,v]) \leq -1$ for all $u,v \in \L_C(-p)$, so we see that
    $$D(\L_C(-p)) \subseteq \L_C(-2p) \coloneqq \{v \in \L_C \mid \ord_p(v) \geq 2\}.$$
    Continuing this way by applying Lemma~\ref{lem:degrees form a semigroup}, we conclude that $\bigcap_{n \in \NN} D^n(\L_C(-p)) = \{0\}$. Since $D^n(\g) \subseteq D^n(\L_C(-p))$ for all $n \in \NN$, we also get $\bigcap_{n \in \NN} D^n(\g) = \{0\}$, proving that $\g$ is not simple (e.g. $D(\g) = [\g,\g]$ is a proper nontrivial ideal of $\g$).
\end{proof}

We are now ready to prove Theorem~\ref{thm:Dixmier}.

\begin{proof}[Proof of Theorem~\ref{thm:Dixmier}]
    If $C = \Aa^1 \nonzero$, then $\L_C$ does not satisfy the Dixmier property, by Example \ref{ex:Dixmier}.
    
    For the converse, suppose $C \not\cong \Aa^1 \nonzero$. Let $\varphi \colon \L_C \to \L_C$ be an injective homomorphism of Lie algebras, and let $\g \coloneqq \im(\varphi)$. Note that $\g \cong \L_C$ is a simple Lie algebra \cite[Proposition 1]{Siebert}. Therefore, if $\g$ has finite codimension in $\L_C$, we must have $\g = \L_C$ by Corollary~\ref{cor:proper subalgebra not simple}, and thus $\varphi$ is an automorphism of $\L_C$ and we are done.

    So, assume $\g$ has infinite codimension in $\L_C$. Since $\g$ is an infinite-dimensional subalgebra of $\L_C$, we let $D \coloneqq D_\g$ be the affine curve associated to $\g$. By Theorem~\ref{thm:main}, $\g$ has finite codimension in $\L_D$. But $\g$ is simple, so $\g = \L_D$ by Corollary~\ref{cor:proper subalgebra not simple}.

    This shows that $\L_D = \g \cong \L_C$. By \cite[Corollary 3]{Siebert}, we have $C \cong D$. Let $\widetilde{C}$ and $\widetilde{D}$ be the smooth projective compactifications of $C$ and $D$, and let $\pi \colon \widetilde{C} \to \widetilde{D}$ be the dominant morphism induced by the inclusion $\kk(\widetilde{D}) = F(\g) \subseteq \kk(\widetilde{C})$. Since $\g$ has infinite codimension in $\L_C$, the map $\pi$ cannot be an isomorphism, as this would imply that $F(\g) = \kk(\widetilde{D}) = \kk(\widetilde{C})$, which is impossible by Theorem~\ref{thm:main}. Therefore, we have $d \coloneqq \deg(\pi) \geq 2$.

    By Corollary~\ref{cor:C maps to D}, $\pi$ restricts to a map of affine curves $\pi' \colon C \to D$. Let $p \in C$ and let $q \coloneqq \pi(p) \in D$. Let $e$ be the ramification index of $\pi$ at $p$. By Lemma~\ref{lem:ramification}, it follows that $\deg_p(w) = e \deg_q(w)$ for all $w \in \L_D$. Now let $w \in \L_D$ be a vector field which does not vanish at $q$, so that $\deg_q(w) = 1$. Then the above implies that $\deg_p(w) = e \deg_q(w) = e$. But $w \in \L_C$ since $\L_D = \g$ is contained in $\L_C$, so $w$ cannot have a pole at $p$, meaning $\deg_p(w) \leq 1$. This forces $e = 1$, so $\pi$ is unramified at $p$. It follows that $\pi'$ is \'etale.

    Now, we have an \'etale map $C \xrightarrow{\pi'} D \cong C$. By Proposition~\ref{prop:Jacobian curves}, we conclude that $\pi'$ is an isomorphism, a contradiction. This finishes the proof.
\end{proof}

\section{Classification of subalgebras of the full Witt algebra}\label{sec:classification Witt}

Recall that the Witt algebra is $W = \L_{\Aa^1 \nonzero} = \Der(\kk[t,t^{-1}]) = \kk[t,t^{-1}]\del$, where $\del = \diff{t}$. In other words, we now set $\widetilde{C} = \PP^1$ and $C = \PP^1 \setminus \{0,\infty\} = \Aa^1 \setminus \{0\}$. In this section, we give a more explicit description of subalgebras of the Witt algebra, analogously to \cite[Theorem 2.8]{BellBuzaglo}, which classifies subalgebras of the one-sided Witt algebra $\WW_1 = \L_{\Aa^1}$.

For the next definition, we note the following: if we have $s \in \kk[t,t^{-1}]$, then $\Der(\kk[s]) = \kk[s]\del_s$, where $\del_s$ is the unique derivation of $\kk[s]$ such that $\del_s(s) = 1$. We can uniquely extend $\del_s$ to a derivation of $\kk(t)$ by identifying $\del_s = \frac{1}{s'}\del$, where $s' = \frac{ds}{dt}$. This lets us embed $\Der(\kk[s])$ into $\Der(\kk(t)) = \kk(t)\del$ as the subalgebra $\frac{1}{s'}\kk[s]\del$.

\begin{definition}\label{def:L(s)}
    Given $s \in \kk[t,t^{-1}]$, define $L(s) \coloneqq W \cap \Der(\kk[s])$, where the intersection takes place in $\kk(t)\del$ upon identifying $\Der(\kk[s]) = \frac{1}{s'}\kk[s]\del$. Given a positive integer $n$, we further define the \emph{$n$-Veronese subalgebra of $W$} to be $\Ver_n \coloneqq \kk[t^n,t^{-n}]t\del$.
\end{definition}

\begin{remark}\label{rem:Veronese}
    The $1$-Veronese subalgebra simply recovers the full Witt algebra (in other words, we have $\Ver_1 = W$). Furthermore, it is easy to see that $\Ver_n = \Der(\kk[t^n,t^{-n}]) \cong W$.
\end{remark}

The following result completely classifies subalgebras of $W$ up to finite codimension.

\begin{theorem}\label{thm:classification up to finite codimension}
    Let $\g$ be an infinite-dimensional subalgebra of $W$. Then either $\g$ has finite codimension in $\Ver_n$ for some positive integer $n$, or there exists $s \in \kk[t,t^{-1}]$ such that $\g$ has finite codimension in $L(s)$. In particular, $\g$ is isomorphic to a subalgebra of finite codimension in $W$ (in the first case) or in $\WW_1$ (in the second case).
\end{theorem}
\begin{proof}
    Let $\widetilde{D} \coloneqq \widetilde{D}_\g$ and $D \coloneqq D_\g$. By Theorem~\ref{thm:main}, we know that $\g$ has finite codimension in $\L_D$. Note that $F(\g) = \kk(\widetilde{D})$ is a subfield of $\kk(t)$, so L\"uroth's theorem implies that $\widetilde{D} \cong \PP^1$. Let $\pi \colon \widetilde{C} \to \widetilde{D}$ be the morphism induced by the inclusion $\kk(\widetilde{D}) \subseteq \kk(\widetilde{C}) = \kk(t)$. By Lemma~\ref{lem:poles in D come from poles in C}, we know that $\pi^{-1}(\widetilde{D} \setminus D) \subseteq \widetilde{C} \setminus C = \{0,\infty\}$, from which it follows that $|\widetilde{D} \setminus D| = 1$ or $2$.
    
    We first consider the case $|\widetilde{D} \setminus D| = 1$, meaning $D \cong \Aa^1$. By Corollary~\ref{cor:C maps to D}, $\O_D$ is a subring of $\O_C = \kk[t,t^{-1}]$, and thus $\O_D = \kk[s]$ for some $s \in \kk[t,t^{-1}]$. Note that
    $$\g \subseteq W \cap \L_D = W \cap \Der(\O_D) = W \cap \Der(\kk[s]) = L(s).$$
    Since $\g$ has finite codimension in $\L_D = \Der(\kk[s])$, it follows that $\g$ has finite codimension in $L(s)$, as required.

    It remains to consider the case $|\widetilde{D} \setminus D| = 2$, meaning $D \cong \Aa^1 \setminus \{0\}$. Applying Corollary~\ref{cor:C maps to D} as above, we know that $\O_D$ is a subring of $\O_C = \kk[t,t^{-1}]$, from which it follows that $\O_D = \kk[s,s^{-1}]$ for some $s \in \kk[t,t^{-1}]$. Since $s$ is invertible in $\kk[t,t^{-1}]$, we deduce that (rescaling $s$ if necessary) $s = t^n$ for some positive integer $n$. Hence,
    $$\L_D = \Der(\O_D) = \Der(\kk[t^n,t^{-n}]) = \Ver_n,$$
    by Remark \ref{rem:Veronese}. Since $\g$ has finite codimension in $\L_D = \Ver_n$, this concludes the proof.

    The final sentence follows upon noting that $\Ver_n \cong W$ and that $L(s)$ has finite codimension in $\Der(\kk[s]) \cong \WW_1$.
\end{proof}

Since Definition~\ref{def:L(s)} does not give an explicit description of $L(s)$, we will explicitly describe $L(s)$, analogously to \cite[Notation 4.11]{Buzaglo}. First, we need the following analog of \cite[Proposition 4.13]{Buzaglo}.

\begin{lemma}\label{lem:g_s}
    Let $s \in \kk[t,t^{-1}]$ be non-constant. Then there exists $g_s \in \kk[t,t^{-1}] \nonzero$ such that $s'g_s \in \kk[s]$ and if $s'g \in \kk[s]$ for some $g \in \kk[t,t^{-1}]$, then $g \in \kk[s]g_s$.
\end{lemma}
\begin{proof}
    Similar to \cite[Proposition 4.13]{Buzaglo}.
\end{proof}

We now give an explicit description of $L(s)$.

\begin{proposition}\label{prop:L(s)}
    Let $s \in \kk[t,t^{-1}]$ and let $g_s \in \kk[t,t^{-1}]$ be as in Lemma~\ref{lem:g_s}. Then $L(s) = \kk[s]g_s \del$.
\end{proposition}
\begin{proof}
    Since $s \in \kk[t,t^{-1}]$ and $g_s \in \kk[t,t^{-1}]$, we certainly have $\kk[s]g_s\del \subseteq W$. Write $s'g_s = h(s)$, where $h(s) \in \kk[s]$. Then
    $$g_s\del = \frac{h(s)}{s'}\del = h(s)\del_s \in \Der(\kk[s]),$$
    where we identify $\del_s = \frac{1}{s'}\del$. It is now clear that $\kk[s]g_s\del = \kk[s]h(s)\del_s \subseteq \Der(\kk[s])$, from which it follows that $\kk[s]g_s\del \subseteq L(s)$.

    Therefore, it remains to show that $L(s) \subseteq \kk[s]g_s\del_s$. To that end, let $v \in L(s) = W \cap \Der(\kk[s])$. Then $v = g\del$ for some $g \in \kk[t,t^{-1}]$ (since $v \in W$) and $v = f(s)\del_s = \frac{f(s)}{s'}\del$ for some $f(s) \in \kk[s]$ (since $v \in \Der(\kk[s])$). It follows that $s'g = f(s) \in \kk[s]$, so $g \in \kk[s]g_s$ by Lemma~\ref{lem:g_s}. We conclude that $v = g\del \in \kk[s]g_s\del$, which finishes the proof.
\end{proof}

To finish this section, we show how every subalgebra of $W$ can be sandwiched between two subalgebras from the family we define below.

\begin{definition}
    Let $s,g \in \kk[t,t^{-1}]$ such that $s'g \in \kk[s]$. We define $L(s,g) \coloneqq \kk[s]g\del$.
\end{definition}

It is easy to see that $L(s,g)$ is a Lie subalgebra of $L(s)$ of finite codimension by the arguments in \cite[Lemma 4.12 and Proposition 4.13]{Buzaglo}. In complete analogy with \cite[Theorem 2.8]{BellBuzaglo}, we now show that any subalgebra of $W$ can be sandwiched between two subalgebras of the form $L(s,g)$, unless the subalgebra is contained in some Veronese of $W$.

\begin{corollary}\label{cor:sandwich}
    Let $\g$ be an infinite-dimensional subalgebra of $W$ which does not have finite codimension in any Veronese subalgebra of $W$. Let $s \in \kk[t,t^{-1}]$ such that $\g$ has finite codimension in $L(s)$ (which exists by Theorem~\ref{thm:classification up to finite codimension}). Then there exists $g \in \kk[t,t^{-1}]$ such that $s'g \in \kk[s]$ and
    $$L(s,g) \subseteq \g \subseteq L(s).$$
\end{corollary}
\begin{proof}
    Similar to \cite[Lemma 1.9]{BellBuzaglo}.
\end{proof}

\begin{remark}
    If $\g$ has finite codimension in $\Ver_n$ (where $n$ is a positive integer), then we still get a similar result to Corollary~\ref{cor:sandwich}: in this case, we can apply the isomorphism $\varphi \colon \Ver_n \to W$ defined by $h(t^n) t\del \mapsto n h(t) t\del$ for $h(t) \in \kk[t]$. Then $\varphi(\g)$ is a finite-codimensional subalgebra of $W$, so we can apply \cite[Proposition 3.2.7]{PetukhovSierra} to deduce that there exist $f \in \kk[t,t^{-1}]$ and $k \in \NN$ such that
    $$f^k W \subseteq \varphi(\g) \subseteq f W,$$
    and thus
    $$f(t^n)^k \Ver_n \subseteq \g \subseteq f(t^n) \Ver_n.$$
\end{remark}

We finish with a question.

\begin{question}
    Can we get a similar explicit description of subalgebras of $\L_C$ for an arbitrary smooth curve $C$?
\end{question}


\begin{thebibliography}{BCC{\etalchar{+}}25}

\bibitem[AS74]{AmayoStewart}
Ralph~K. Amayo and Ian Stewart, \emph{Infinite-dimensional {L}ie algebras}, Noordhoff International Publishing, Leyden, 1974.

\bibitem[BB25a]{BellBuzaglo2}
Jason Bell and Lucas Buzaglo, \emph{Enveloping algebras of derivations of commutative and noncommutative algebras}, Int. Math. Res. Not. IMRN (2025), no.~17, Paper No. rnaf265, 11 pp.

\bibitem[BB25b]{BellBuzaglo}
\bysame, \emph{Maximal dimensional subalgebras of general {C}artan-type {L}ie algebras}, Bull. Lond. Math. Soc. \textbf{57} (2025), no.~2, 605--624.

\bibitem[BCC{\etalchar{+}}25]{BagchiChakrabortyChakraborttyFredenhagenGrumillerPandit}
Arjun Bagchi, Pronoy Chakraborty, Shankhadeep Chakrabortty, Stefan Fredenhagen, Daniel Grumiller, and Priyadarshini Pandit, \emph{Boundary {C}arrollian conformal field theories and open null strings}, Phys. Rev. Lett. \textbf{134} (2025), no.~7, Paper No. 071604, 8 pp.

\bibitem[BHP{\etalchar{+}}26]{BCCA}
Lucas Buzaglo, Xiao He, Tuan~Anh Pham, Haijun Tan, Girish~S. Vishwa, and Kaiming Zhao, \emph{On the boundary {C}arrollian conformal algebra}, 2026, to appear in Lett. Math. Phys.

\bibitem[Buz23]{Buzaglo}
Lucas Buzaglo, \emph{Enveloping algebras of {K}richever-{N}ovikov algebras are not {N}oetherian}, Algebr. Represent. Theory \textbf{26} (2023), no.~5, 2085--2111.

\bibitem[CGLZ17]{CoxGuoLuZhao}
Ben Cox, Xiangqian Guo, Rencai Lu, and Kaiming Zhao, \emph{Simple superelliptic {L}ie algebras}, Commun. Contemp. Math. \textbf{19} (2017), no.~3, Paper No. 1650032, 22 pp.

\bibitem[Du04]{Du}
Hong Du, \emph{Endomorphisms of {L}ie algebra {$F[t]\frac d{dt}$}}, J. Syst. Sci. Complex. \textbf{17} (2004), no.~1, 143--146.

\bibitem[Fia83]{Fialowski}
Alice Fialowski, \emph{On the classification of graded {L}ie algebras with two generators}, Vestnik Moskov. Univ. Ser. I Mat. Mekh. (1983), no.~2, 62--64.

\bibitem[Gra79]{Grabowski}
J.~Grabowski, \emph{Isomorphisms and ideals of the {L}ie algebras of vector fields}, Invent. Math. \textbf{50} (1978/79), no.~1, 13--33.

\bibitem[GS64]{GuilleminSternberg}
Victor~W. Guillemin and Shlomo Sternberg, \emph{An algebraic model of transitive differential geometry}, Bull. Amer. Math. Soc. \textbf{70} (1964), 16--47.

\bibitem[KN87]{KricheverNovikov}
Igor~Moiseevich Krichever and S.~P. Novikov, \emph{Algebras of {V}irasoro type, {R}iemann surfaces and strings in {M}inkowski space}, Funktsional. Anal. i Prilozhen. \textbf{21} (1987), no.~4, 47--61, Paper No. 96.

\bibitem[Mat86a]{Mathieu2}
Olivier Mathieu, \emph{Classification des alg\`ebres de {L}ie gradu\'{e}es simples de croissance {$\leq 1$}}, Invent. Math. \textbf{86} (1986), no.~2, 371--426.

\bibitem[Mat86b]{Mathieu1}
\bysame, \emph{Sur un probl\`eme de {V}. {G}. {K}ac: la classification de certaines alg\`ebres de {L}ie gradu\'{e}es simples}, J. Algebra \textbf{102} (1986), no.~2, 505--536.

\bibitem[Mat92]{Mathieu3}
\bysame, \emph{Classification of simple graded {L}ie algebras of finite growth}, Invent. Math. \textbf{108} (1992), no.~3, 455--519.

\bibitem[Mat26]{Mathieu}
\bysame, \emph{Weakly noetherian {L}ie algebra and the {S}ierra-{W}alton conjecture}, 2026, arXiv:\href{https://arxiv.org/abs/2605.18116}{\texttt{2605.18116}}.

\bibitem[MR01]{McConnellRobson}
J.~C. McConnell and J.~C. Robson, \emph{Noncommutative {N}oetherian rings}, revised ed., Graduate Studies in Mathematics, vol.~30, American Mathematical Society, Providence, RI, 2001, With the cooperation of L. W. Small.

\bibitem[PS23]{PetukhovSierra}
Alexey~V. Petukhov and Susan~J. Sierra, \emph{The {P}oisson spectrum of the symmetric algebra of the {V}irasoro algebra}, Compos. Math. \textbf{159} (2023), no.~5, 933--984.

\bibitem[Sch93]{SchlichenmaierDegenerations}
Martin Schlichenmaier, \emph{Degenerations of generalized {K}richever-{N}ovikov algebras on tori}, J. Math. Phys. \textbf{34} (1993), no.~8, 3809--3824.

\bibitem[Sch14]{Schlichenmaier}
\bysame, \emph{Krichever-{N}ovikov type algebras}, De Gruyter Studies in Mathematics, vol.~53, De Gruyter, Berlin, 2014, Theory and applications.

\bibitem[Sie96]{Siebert}
Thomas Siebert, \emph{Lie algebras of derivations and affine algebraic geometry over fields of characteristic {$0$}}, Math. Ann. \textbf{305} (1996), no.~2, 271--286.

\bibitem[Sil09]{Silverman}
Joseph~H. Silverman, \emph{The arithmetic of elliptic curves}, second ed., Graduate Texts in Mathematics, vol. 106, Springer, Dordrecht, 2009.

\bibitem[SW14]{SierraWalton}
Susan~J. Sierra and Chelsea Walton, \emph{The universal enveloping algebra of the {W}itt algebra is not noetherian}, Adv. Math. \textbf{262} (2014), 239--260.

\end{thebibliography}
\end{document}